\documentclass{amsart}
\usepackage[left=1.25in, right=1.25in, top=1.25in, bottom=1.25in]{geometry}
\usepackage{layout}

\usepackage{amssymb, amsfonts,amsthm,amsmath,latexsym,mathrsfs}
\usepackage[all]{xy}

\def\bt{\begin{thm}}
\def\et{\end{thm}}
\def\bl{\begin{lem}}
\def\el{\end{lem}}
\def\bd{\begin{defi}}
\def\ed{\end{defi}}
\def\bc{\begin{cor}}
\def\ec{\end{cor}}
\def\bp{\begin{proof}}
\def\ep{\end{proof}}
\def\br{\begin{rem}}
\def\er{\end{rem}}

\newtheorem{thm}{Theorem}[section]
\newtheorem{prop}[thm]{Proposition}
\newtheorem{lem}[thm]{Lemma}
\newtheorem{defn}[thm]{Definition}
\newtheorem{example}[thm]{Example}
\newtheorem{rem}[thm]{Remark}
\newtheorem{cor}[thm]{Corollary}

\numberwithin{equation}{section}

\newcommand{\C}{\Bbb{C}}
\newcommand{\Z}{\widetilde{Z}}
\newcommand{\cohomology}{H^{1,1}(X,\mathbb{R})}

\newcommand{\la}{\langle}
\newcommand{\ra}{\rangle}
\newcommand{\U}{\mathscr{U}}
\newcommand{\Ck}{\mathscr{C}_k}
\newcommand{\Cmk}{\mathscr{C}_{m-k+1}}

\newcommand{\ommk}{\omega^{m-k+1}}
\newcommand{\p}{\Bbb{P}}
\newcommand{\prob}{{\bf{P}}}
\newcommand{\probn}{{\bf{P}_n}}

\newcommand{\bthm}{\begin{thm}}
\newcommand{\ethm}{\end{thm}}
\newcommand{\bstp}{\begin{stp}}
\newcommand{\estp}{\end{stp}}
\newcommand{\blemma}{\begin{lemma}}
\newcommand{\elemma}{\end{lemma}}
\newcommand{\bprop}{\begin{prop}}
\newcommand{\eprop}{\end{prop}}
\newcommand{\bpf}{\begin{pf}}
\newcommand{\epf}{\end{pf}}
\newcommand{\bdefn}{\begin{defn}}
\newcommand{\edefn}{\end{defn}}
\newcommand{\brk}{\begin{rmrk}}
\newcommand{\erk}{\end{rmrk}}
\newcommand{\bcrl}{\begin{crl}}
\newcommand{\ecrl}{\end{crl}}

\title{Equidistribution of zeros of random holomorphic sections}
\author{Turgay Bayraktar}
\date{\today}

\address{Mathematics Department, Syracuse University, NY, 13205 USA}
\email{tbayrakt@syr.edu}

\begin{document}

\begin{abstract}
We study asymptotic distribution of zeros of  random holomorphic sections of high powers of positive line bundles defined over projective homogenous manifolds. We work with a wide class of  distributions that includes real and complex Gaussians. %We prove that normalized simultaneous zero currents of i.i.d. random holomorphic sections, orthonormalized on a regular compact set, converges almost surely to the expected limit distribution.  
%As a special case, we obtain asymptotic zero distribution of ensembles of orthogonal polynomials. 
As a special case, we obtain asymptotic zero distribution of multivariate complex polynomials given by linear combinations of orthogonal polynomials with i.i.d. random coefficients. Namely, we prove that normalized zero measures of m i.i.d random polynomials, orthonormalized on a regular compact set $K\subset \Bbb{C}^m,$ are almost surely asymptotic to the equilibrium measure of $K$.
\end{abstract}

\maketitle
%\tableofcontents
\section{Introduction}
 In this paper we study limit distribution of zeros of random holomorphic sections of high powers $L^{\otimes n}$ of a positive holomorphic line bundle $L$ defined over a projective manifold $X.$ %We consider ensembles of orthogonal global holomorphic sections in $H^0(X,L^{\otimes n}),$ orthonormalized on a regular weighted compact set $(K,q)$ (see \S \ref{set} for details). Orthogonal polynomials ensembles as well as their spherical counterpart (spherical harmonics) arise as a special case for the appropriate choice of $L\to X.$ 
 Given a regular compact set $K\subset X,$ a continuous weight function $q:K\to \Bbb{R}$ and a smooth positively curved hermitian metric $h$ on $L\to X$ one can define a scalar $L^2$-product on $H^0(X,L^{\otimes n})$ (see \S \ref{inner} for details).  For a fixed orthonormal basis of $H^0(X,L^{\otimes n})$ with respect to this $L^2$-product, we consider linear combinations in which coefficients are i.i.d. real or complex random variables. Multivariate polynomials with random coefficients (with respect to a suitable orthonormal basis) as well as their spherical counterpart (spherical harmonics) arise as a special case for the appropriate choice of $L\to X.$ Zero distribution of random holomorphic sections with i.i.d. Gaussian coefficients was studied by Shiffman and Zelditch \cite{SZ,SZ1,SZ2}. The setup of the later papers correspond here to the special case $K=X$ and $q=0.$ In the present setting, we allow coefficients to be i.i.d random variables of bounded density with logarithmically decaying tails (see (\ref{h2})). Our main result (Theorem \ref{main}) asserts that in any codimension normalized zero currents of i.i.d. random holomorphic sections are almost surely asymptotic to external powers $T_{K,q}\wedge\dots \wedge T_{K,q}$ of an extremal current $T_{K,q}$ associated to $(L,h,K,q).$ In particular, when $K=X$ and $q=0$ the current $T_{K,q}$ coincides with the curvature form $\omega=c_1(h)$ of the smooth positively curved hermitian metric $h$ on $L,$ hence we recover \cite[Theorem 1]{SZ}. Therefore, these results can also be considered as a global universality of distribution of zeros in the sense that they extend earlier known results in the setting of Gaussian random holomorphic sections (as well as polynomials) to a more general setting.    \\ \indent
  If $X:=\p^m$ is the complex projective space and $L$ is the hyperplane bundle $\mathcal{O}(1)\to \p^m$ with a suitable choice of a hermitian metric (see \S \ref{BMp}) the present geometric setting reduces to random multivariable holomorphic polynomials. A classical result due to Kac and Hammersley \cite{Kac,Ham} asserts that normalized zeros of Gaussian random univariate polynomials of large degree tend to accumulate on the unit circle. Recently, Ibragimov and Zaporozhets \cite{IZ} provided a necessary and sufficient condition for global universality of zero distribution for Kac ensembles (see (\ref{ic})). More generally, zero distribution of random multivariate polynomials has been studied in  \cite{Bloom1,BloomS,Shiffman} and it was proved that normalized simultaneous zero measures of i.i.d. Gaussian random polynomial systems, orthonormalized on a regular compact set $K\subset\Bbb{C}^m,$ is asymptotic to the equilibrium measure of $K$. However, beyond the Gaussian ensembles not much is known about the asymptotic zero distribution of random polynomials in higher dimensions. 
More recently, Bloom and Levenberg \cite{BloomL} studied this problem for absolutely continuous distributions that have polynomially decaying tails and they proved that expected normalized zero measures of i.i.d. random polynomials converge to equilibrium measure of $K$ as their degree grow. Moreover, they posed the almost everywhere convergence of normalized zero measures of i.i.d random polynomials to the expected distribution as an open problem. We address this question in  the affirmative for a more general class of distributions that have logarithmically decaying tails (Theorem \ref{poly}).\\ \indent
 Recently, authors of \cite{CMM} obtained another generalization of \cite[Theorem 5.2]{BloomL}. Namely, for a given bounded positive singular hermitian metric $h=e^{-\varphi}$ it is proved in \cite[Corollary 5.6]{CMM} that normalized zero currents of Gaussian random holomorphic sections in $H^0(X,L^{\otimes n})$ converge almost surely to the curvature current $\omega+dd^c\varphi.$ We remark that the scalar $L^2$-product used in \cite{CMM} is different than the one used in \cite{BloomL} and this paper.
  \subsection{The setting}\label{set} Let $X$ be a projective manifold of complex dimension $m$ and $L\to X$ be a positive holomorphic line bundle. We also let $\omega$ be a smooth positive $(1,1)$ form representing $c_1(L),$ the first Chern class of $L$. Recall that an usc function $\varphi\in L^1(X)$ is called $\omega$-psh if $\omega+dd^c\varphi\geq0$ in the sense of currents. We denote the set of all $\omega$-psh functions by $PSH(X,\omega)$.\\ \indent 
 Given a non-pluripolar compact set $K\subset X$ and a continuous weight function $q:K\to \Bbb{R}$  \textit{weighted global extremal function} $V_{K,q}^*$ of $K$ is defined to be usc regularization of 
 \begin{equation}\label{def}
 V_{K,q}:=\sup\{\varphi\in PSH(X,\omega):  \varphi\leq q\ \text{on}\ K\}.
 \end{equation}
By definition we say that $(K,q)$ is a \textit{regular weighted compact set} if $V_{K,q}$ is continuous on $X.$ Throughout this work, we assume that $V_{K,q}$ is continuous and hence $V_{K,q}=V_{K,q}^*.$ Recall that a sufficient condition for continuity of $V_{K,q}$ is local regularity of $K$ (\cite{Siciak,BBN}). It is well-known that $T_{K,q}:=\omega+dd^c V_{K,q}$ defines a positive closed $(1,1)$ current representing the class $c_1(L)$ in $\cohomology.$ Since $T_{K,q}$ has locally bounded potentials, it follows from Bedford-Taylor theory \cite{BT2} that the exterior powers $T_{K,q}^k:=T_{K,q}\wedge\dots\wedge T_{K,q}$ (k-times) are well defined positive closed bidegree $(k,k)$ currents for each $1\leq k\leq \dim X.$ In particular, the top degree self-intersection (after normalizing) defines a probability measure on $X.$
 
Given a measure $\tau$ on $K$ and a smooth positively curved metric $h$ on $L,$ one can define an $L^2$-norm on $H^0(X,L^{\otimes n})$ by $$\|s\|_{L^2_{q,\tau}}:=\big(\int_K\|s(x)\|_{h_n}^2e^{-2nq(x)}d\tau(x)\big)^{\frac12}.$$ % scalar  inner product (see  (\ref{inner})) on the space of global holomorphic sections $H^0(X,L^{\otimes n}).$ 
We fix an orthonormal basis $\{S^n_j\}_{j=1}^{d_n}$ for $H^0(X,L^{\otimes n})$ induced by this norm. Throughout this paper, we assume that the point-wise norm of restriction of Bergman kernel to the diagonal 
$$ \|S_n(x,x)\|_{h_n}:=\sum_{j=1}^{d_n}\|S_j^n(x)\|^2_{h_n}$$
has sub-exponential growth (see \S \ref{BM} for details). Such measures $\tau$, called  Bernstein-Markov (BM) measures in the literature, and they always exist on regular weighted compact sets (\cite{NgZ,BBN}).  We remark that although the inner product on $H^0(X,L^{\otimes n})$ depends on the choice of the Bernstein-Markov measure $\tau,$ the limiting distribution of zeros does not depend on it (see Theorem \ref{main}).

Every $s_n\in H^0(X,L^{\otimes n})$ can be uniquely written as $$s_n=\sum_{j=1}^{d_n}a^{(n)}_jS^n_j.$$ We assume that $a^{(n)}_j$ are real or complex valued i.i.d random variables whose distribution law is of the form $\prob:=\phi(z) d\lambda(z)$ satisfying 
\begin{equation}\label{h1}
0\leq \phi(z)\leq C\ \text{for some}\ C>0
\end{equation}
\begin{equation}\label{h2}
\prob\{z\in\C:\log |z|>R\}= O(R^{-\rho})\ \text{as}\ R\to \infty
\end{equation}
 where $\lambda$ is the Lebesgue measure on $\C$ and $\rho>\dim_{\Bbb{C}} X+1.$ If $\phi$ is a function defined on real numbers then we replace $\phi(z)d\lambda$ by $\phi(x)dx.$ We remark that our setting includes standard real and complex Gaussian distributions. %of mean zero and variance one. Recall that, in the later case the tail decay of the integrals is of order $e^{-R^2}.$ 
  \\ \indent We identify $\mathcal{S}_n=H^0(X,L^{\otimes n})$ with $\Bbb{C}^{d_n}$ where $d_n:=h^0(X,L^{\otimes n})$ and endow it with the $d_n$-fold product probability measure $\mu_n$ induced by $\prob$. We also consider $\mathcal{S}_{\infty}:=\prod_{n=1}^{\infty}\mathcal{S}_n$ as a probability space endowed with the product measure $\mu:=\prod_{n=1}^{\infty} \mu_n.$ 
 For a system $S^k_n=(s_n^1,s_n^2,\dots,s_n^k)\in \mathcal{S}_n^k$ of i.i.d. random holomorphic sections with $1\leq k\leq m,$ we denote their common zero locus by
 $$Z_{S^k_n}:=\{x\in X: s_n^1(x)=\dots=s_n^k(x)=0\}$$ and define the normalized zero currents $$\Z_{S^k_n}:=\frac{1}{n^k}[Z_{S^k_n}]$$ where $[Z_{S^k_n}]$ denotes the current of integration along the variety $Z_{S^k_n}.$ Then the \textit{expected zero current} is defined  by 
 $$\la \Bbb{E}[\Z_{S^k_n}],\Phi\ra:=\int_{\mathcal{S}_n^k}\la \Z_{S^k_n},\Phi\ra d\mu_n^{k}(S^k_n)$$
 where $\Phi$ is a bidegree $(m-k,m-k)$ test form on $X$ and $\mu_n^{k}=\mu_n\times\dots\times \mu_n$ is the $k$-fold product measure. 
  \begin{thm}\label{main}
Let $L\to X$ be a positive holomorphic line bundle over a projective manifold $X$ and $(K,q)$ be a regular weighted compact set. Then for $1\leq k\leq \dim_{\Bbb{C}} X$
$$\Bbb{E}[\Z_{S^k_n}]\to T_{K,q}^k$$
in the sense of currents as $n\to \infty.$ Moreover, if the ambient space $X$ is complex homogeneous then almost surely 
$$\Z_{S^k_n}\to T_{K,q}^k$$ in the sense of currents as $n\to \infty.$
 \end{thm}
% \\ \indent
  The proof of Theorem \ref{main} is based on induction on bidegrees. To prove almost everywhere convergence for $k=1,$ we use Bergman kernel asymptotics together with Kolmogorov's strong law of large numbers. The later requires a variance estimate (Lemma \ref{var1}). To this end we make use of exponential estimates for qpsh functions which can be considered as a global version of uniform Skoda integrability theorem. Finally, we use extremal property of $V_{K,q}$ to dominate quasi-potentials of limit points of random sequence of zero currents $\{\Z_{s_n}\}_{n\geq1}.$ In higher bidegrees, we work with super-potentials of positive closed currents. Recall that the super-potentials of positive closed currents were introduced by Dinh and Sibony \cite{DS11,DS10} which extends the notion of quasi-potentials of positive closed $(1,1)$ currents.\\ \indent
Recall that the randomization in \cite{SZ} is obtained by taking $K=X$ and endowing $\Bbb{P}H^0(X,L^{\otimes n})$ with the Fubini-Study volume form. In particular, if the ambient space is homogenous our results generalizes that of \cite{SZ,Shiffman}. Note that our proof is partly based on resolution of $\partial\bar{\partial}$-equations with qualitative estimates (see Theorem \ref{DS}) which requires the ambient space to be homogenous. It would be interesting to know if one can prove such equidistribution results on arbitrary projective manifolds.  More recently, Dinh and Sibony \cite{DS3} studied this problem by endowing $\Bbb{P}H^0(X,L^{\otimes n})$ with moderate measures. Recall that  Monge-Amp\`ere measure of a H\"{o}lder continuous qpsh function is among the examples of moderate measures (see \cite{DNS} for details).%\\ \indent

\subsection{Multivariate orthogonal polynomial ensembles}\label{BMp}
 In this section we explain how multivariate complex polynomials arise as a special case in the above general geometric setting. Let $K\subset \Bbb{C}^m$ be a compact set and $q:K\to\Bbb{R}$ be a continuous function. Recall that the \textit{weighted global extremal function} $V_{K,q}^*$ is defined as usc regularization of
  $$V_{K,q}:=\sup\{u\in Psh(\Bbb{C}^m):u\leq q\ \text{on}\ K\ \text{and}\ u(z)\leq\log^+\|z\|+O(1)\ \text{as}\ \|z\|\to \infty\}.$$
  In the unweighted case (i.e. $q\equiv0$), we write $V_K$ for short. We say that $K$ is \textit{regular} if $V_K$ is continuous on $\Bbb{C}^m.$ A set $K$ is called \textit{locally regular} at $z$ if $K\cap B(z,r)$ is regular for every $r>0$ where $B(z,r)$ denotes the ball centered at $z$ with radius $r.$ If $K$ is locally regular at every $z\in K$ we say that $K$ is locally regular. Polydisc and round sphere in $\Bbb{C}^m$ are among the examples of locally regular compact sets. It follows from \cite[Prop. 2.16]{Siciak} that  for a locally regular compact set $K$, the function $V_{K,q}$ is continuous for every continuous weight $q$ and hence $V_{K,q}=V_{K,q}^*$.
   %Next, we denote by $$\Phi_{K,q}^n:=  \sup \{|p(z)|: p\in \mathcal{P}_n \ \text{and}\ \max_{z\in K}|p(z)e^{-nq(z)}|\leq 1\}$$ where $ \mathcal{P}_n$ denote the set of polynomials on $\Bbb{C}^m$ of degree at most $n.$
  %It follows from seminal work of Siciak (\cite{Siciak}) and Zaharjuta (\cite{Zaharjuta}) that $\frac1n\log\Phi_{K,q}^n$ converges to $V_{K,q}.$\\ \indent 
Then  by Bedford-Taylor theory \cite{BT2} the exterior powers 
    $$(\frac{i}{\pi}\partial\overline{\partial}V_{K,q})^k:=\frac{i}{\pi}\partial\overline{\partial}V_{K,q}\wedge\dots \wedge \frac{i}{\pi}\partial\overline{\partial}V_{K,q}$$ 
    are well defined positive closed bidegree $(k,k)$ currents for $1\leq k\leq m$. In particular, the top degree intersection $\mu_{K,q}:=(\frac{i}{\pi}\partial\overline{\partial}V_{K,q})^m$ is a probability measure. We write $\mu_K$ for short when $q\equiv0.$ The measure $\mu_K$ is called the \textit{equilibrium measure} of $K$ in the literature. We refer the reader to the text \cite{SaffTotik} for details and background on weighted pluripotential theory (see  \cite[Appendix B]{SaffTotik} for multivariate case). \\ \indent
Recall that complex projective space $X:=\p^m$ is defined as the quotient $\Bbb{C}^{m+1}-\{0\}/\Bbb{C}^*$ and its elements are represented by homogenous coordinates $[Z_0:\dots:Z_{m}]$ where $(Z_0,\dots,Z_{m})\to [Z_0:\dots:Z_{m}]$ is the standard projection $\pi:\Bbb{C}^{m+1}-\{0\} \to \p^m.$ The fibers of $\pi$ are complex lines in $\Bbb{C}^{m+1}$ and hence defines a line bundle (called tautological bundle) over $\p^m.$ The dual of this line bundle is called hyperplane bundle which we denote by $L:=\mathcal{O}(1).$ Then we can identify $\Bbb{C}^m$ with the affine piece $U_0=\{Z_0\not=0\}$ via the embedding $z\to [1:z].$ Note that by definition of $\mathcal{O}(1)$ homogenous coordinates $Z_i$ define sections of $\mathcal{O}(1).$ As a result we may identify $H^0(\p^m,\mathcal{O}(n))$ with homogenous polynomials in $m+1$ variables of degree $n.$ In particular, restricting them on the affine piece $\Bbb{C}^m\simeq U_0,$ the space $H^0(\p^m,\mathcal{O}(n))$ gets identified with the space of polynomials on $\Bbb{C}^m$ of degree at most $n.$ Next, we endow $\mathcal{O}(1)$ with the Fubini-Study metric $h:=h_{FS}$ which can be represented by the weight function $\frac12\log (1+\|z\|^2)$ on the affine piece $\Bbb{C}^m.$ We also denote the Chern form of $h_{FS}$ by $\omega_{FS}.$ Then the function $V_{K,q}-\frac12\log(1+\|z\|^2)$ extends uniquely to a $\omega_{FS}$-psh function on $\p^m$ and the extension coincides with (\ref{def}) (see \cite{GZ}). Hence, applying Theorem \ref{main} we obtain:
\begin{thm}\label{poly}
Let $K\subset \Bbb{C}^m$ be a locally regular compact set, $q:K\to \Bbb{R}$ be a continuous weight  function. Then for $1\leq k\leq m$
 $$\Bbb{E}[\Z_{f_n^1,\dots,f_n^k}] \to (\frac{i}{\pi}\partial\overline{\partial}V_{K,q})^k$$
 in the sense of currents as $n\to \infty.$ Moreover, almost surely
 $$\Z_{f_n^1,\dots,f_n^k} \to (\frac{i}{\pi}\partial\overline{\partial}V_{K,q})^k$$
in the sense of currents as $n\to \infty.$ 
\end{thm} 
   Finally, we remark that conditions (\ref{h1}) and (\ref{h2}) with $\rho>2$ are not optimal in the case of $X=\Bbb{P}^1$ and  $K=S^1$ the unit circle in $\Bbb{C}$ with $q\equiv 0.$ It was observed in \cite{IZ} that the normalized zero measures $\Z_{f_n}$ of i.i.d random polynomials $f_n(z)=\sum_{j=0}^na_jz^j$ converges almost surely to the Lebesgue measure $\frac{1}{2\pi}d\theta$  if and only if the distribution law of $a_j$ satisfies
   \begin{equation}\label{ic}
   \int_{\Bbb{C}}\log(1+|a|)d{\bf{P}}(a)<\infty.
   \end{equation}
In our setting, an easy computation shows that (\ref{ic}) holds for $\rho>1$ in (\ref{h2}) %Moreover, by the argument \cite[pp. 6]{IZ} the assumption (\ref{ic}) is a necessary condition for $E[\Z_{f_n}]\to \frac{1}{2\pi}d\theta$ weak* as $n\to \infty.$    

%In the special case, $K\subset \Bbb{C}^m$ is a locally regular compact set with a continuous weight function $q:K\to \Bbb{R}$ we can consider $\Bbb{C}^m$ as an open chart in $\Bbb{P}^m$ where $\p^m$ is endowed with hyperplane bundle $L:=\mathcal{O}(1)$ and Fubini-Study metric. Thus, the elements of $H^0(\Bbb{P}^m,\mathcal{O}(n))$ can be identified with the homogenous polynomials of degree $n$ in $m+1$ variables and their restriction to $\Bbb{C}^m$ can be identified with ensemble of random polynomials $\mathcal{P}_n$ endowed with the product probability measure induced from $H^0(\p^m, \mathcal{O}(n)).$ The restriction of weighted global extremal function is given by
%$$V_{K,q}:=\{\varphi\in Psh(\Bbb{C}): \varphi=\log |z|+O(1)\ \text{as}\ |z|\to \infty \ \text{and}\ \varphi\leq q\ \text{on}\ K\}.$$ 
%For a $k$-tuple of random i.i.d. polynomials $F_n:=(f_1,f_2,\dots,f_n)$ we denote the common zero set $$Z_{F_n}=\{z\in\Bbb{C}^m:f_1(z)=f_2(z)=\dots=f_k(z)=0\}$$ and we denote the normalized zero currents by $\Z_{F_n}:\frac{1}{n^k}[Z_{F_n}].$ %As a consequence of Theorem \ref{main} we obtain

%\begin{thm}
%Let $K\subset \Bbb{C}^m$ be a locally regular compact set with a continuous weight function $q:K\to \Bbb{R}$. Then
%$$E[\Z_{F_n}]\to (dd^cV_{K,q})^k$$ in the sense of currents as $n\to \infty.$
%Moreover, almost surely $$\Z_{F_n}\to (dd^cV_{K,q})^k$$ in the sense of currents as $n \to \infty.$

%\end{thm}
\subsection{Examples}
In this section we provide ensembles of random polynomials for which Theorem \ref{poly} applies. We let $\mathcal{P}_n$ denote the space of polynomials of degree at most $n.$
\begin{example}[Kac Ensemble] Let $K:=\{(z_1,\dots,z_m): \max|z_j|\leq 1\}$ be the unit polydisc in $\Bbb{C}^m$. Then the unweighted global extremal function of $K$ is $V_K:=\max\log^+|z_j|$ and the equilibrium measure $\mu_K=\frac{1}{2\pi}d\theta_1\dots\frac{1}{2\pi}d\theta_m$ where $d\theta$ is the angular measure on the unit circle $S^1.$ Note that equilibrium measure is a BM measure and supported on the torus $(S^1)^m.$ In this case, the monomials $z^J:=z_1^{j_1}\dots z_m^{j_m}$ with $|J|\leq n$ form an orthonormal basis for $\mathcal{P}_n$ with respect to $L^2(\mu_K)$ and a random polynomial is of the form
$$f_n(z)=\sum_{|J|\leq n}a_Jz^J.$$
\end{example}
\begin{example} Let $\Omega$ be a bounded open set in $\Bbb{C}^m$ with $\mathcal{C}^1$ boundary. Then the set $K:=\overline{\Omega}$ is regular \cite[5.3.13]{Klimek}. It follows from \cite{NgZ} and \cite[5.6.7]{Klimek} that the equilibrium measure $\mu_K$ is a BM measure. For instance, if $K=\{\|z\|\leq 1\}$ is the unit ball in $\Bbb{C}^m$ then the (unweighted) global extremal function of $K$ is $V_K(z)=\log^+\|z\|$ and the equilibrium measure $d\sigma=(\frac{i}{\pi}\partial\overline{\partial}\log^+\|z\|)^m$ is the surface area measure on the unit sphere $S^{2m-1}.$ Moreover, the scaled monomials 
 $$c_Jz^J:=(\frac{(j_1+\dots+j_m+m-1)!}{(m-1)!j_1!\dots j_m!})^{\frac12}z_1^{j_1}\dots z_m^{j_m}$$
form an orthonormal basis for $\mathcal{P}_n$ (see \cite[\S4]{BloomS}) with respect to $L^2(\mu_K)$ where we use the multidimensional notation $J=(j_1,\dots, j_m).$ Thus, a random polynomial in this setting is of the form
$$f_n(z)=\sum_{|J|\leq n}a_Jc_Jz^J$$
 Then, applying Theorem \ref{poly} we obtain scaling limits of zeros of random polynomials orthonormalized on $S^{2m-1}.$ \\ \indent
In particular, if $\Omega\subset\Bbb{C}$ is a simply connected domain with real analytic boundary then $\tau=|dz|$ satisfies the BM inequality. Corresponding asymptotic distribution of zeros of Gaussian univariate polynomials are studied in \cite{SZ3}.\\
\end{example}
\begin{example}[Elliptic Polynomials] Let $X=K=\Bbb{P}^m$  and $h:=h_{FS}$ Fubini-Study metric on the hyperplane bundle $L:=\mathcal{O}(1).$ As explained above $H^0(X,\mathcal{O}(n))$ can be identified the set of homogenous polynomials of degree $n$ in $m+1$ variables. Then $S_J^n:=z_0^{j_0}\dots z_m^{j_m}$ form an orthogonal basis for $H^0(X,\mathcal{O}(n))$ where $[z_0:\dots:z_m]$ denotes the homogeneous coordinates on $\Bbb{P}^m.$ Moreover, an easy computation shows that 
$$\|S^n_J\|=(\frac{m!(n-|J|)!j_0!\dots j_m!}{(n+m)!})^{\frac12}$$
(see \cite[\S4]{SZ} for details). Applying Theorem \ref{main} we obtain that almost surely 
$$\Z_{S^k_n}\to\omega_{FS}^k$$ where $\omega_{FS}$ denotes Fubini-Study form on $\p^m.$\\ \indent
In particular, if $m=1$ then the restriction of $S_n$'s to $\Bbb{C}$ gives 
$$f_n(z)=\sum_{j=1}^na_j\sqrt{(n+1){n\choose j}}  z^j$$
which are called elliptic polynomials in the literature. In this setting, almost surely
$$\frac1n\sum_{\{z:f_n(z)=0\}}\delta_z\to \frac{1}{\pi}\frac{dz}{(1+|z|^2)^2}$$ 
weak* as $n\to\infty.$
\end{example}
The next example, motivated from theory of toric varieties, allows us to describe asymptotic distribution of systems of random sparse polynomials:
\begin{example}[Sparse Polynomials] 
By an integral polytope we mean convex hull of a non-empty finite set in $\Bbb{Z}^m.$ Let $P\subset \Bbb{R}^m$ be a Delzant integral polytope (see \cite{Delzant} for precise description). It is well know that such a $P$ induces a triple $(X_P,L_P,h_P)$ where $X_P$ is an $m$ complex dimensional projective toric manifold which contains the complex torus $(\Bbb{C}^*)^m$ as a Zariski dense open set such that the action of $(\Bbb{C}^*)^m$ on it self extends to a $(\Bbb{C}^*)^m$-action on $X_P.$ Moreover, $L_P\to X_P$ is a positive holomorphic line bundle endowed with an invariant (under the action of real torus) smooth positive hermitian metric $h_L$ (see \cite{Delzant, Abreu} for details). Furthermore, we can identify $H^0(X_P,L_P^{\otimes n})$ with the space $Poly(nP)$ spanned by multi-monomials $z^J$ with multi-index $J\in nP\cap \Bbb{Z}^m$ where $nP$ denotes the scaled polytope. Now, letting $K\subset(\Bbb{C}^*)^m$ be a compact set as in Theorem \ref{main} and taking $q\equiv0,$ we obtain asymptotic distribution of i.i.d. systems of random sparse polynomials $(f_n^1,\dots,f_n^m)$ such that $f_n^i\in Poly(nP)$ for $i=1,\dots,m$.   
\end{example}
\section*{Acknowledgement}
I am grateful to Norm Levenberg for many stimulating conversations on the content of this work. I also want to thank Thomas Bloom and Norm Levenberg for their suggestions on an earlier draft. Finally, I would like to thank the anonymous referee for his comments which improve the presentation of this article.

\section{preliminaries}

\subsection{Positive closed currents and super-potentials}
Let $X$ be a connected compact complex manifold and $Aut(X)$ denote the group of holomorphic automorphisms of $X.$ Following Bochner and Montgomery \cite{BM}, $Aut(X)$ is a complex Lie group. We say that $X$ is \textit{homogeneous} if $Aut(X)$ acts transitively on $X.$ In the sequel, we let $(X,\omega)$ be a compact K\"ahler homogeneous manifold of dimension $m$ and $\omega$ is a fixed K\"ahler form. It follows from \cite{BR} that $X$ is a direct product of a complex torus and a projective rational manifold. In particular, complex projective space $\Bbb{P}^m$ and $(\Bbb{P}^1)^m$ are among the examples of such manifolds.  
%Recall that a complex manifold $X$ is called homogenous if the automorphism group of $X$ acts transitively on $X$. Let $(X,\omega)$ be a compact K\"ahler homogenous manifold of dimension $m$ and $\omega$ be a fixed K\"ahler form on $X.$ Note that complex projective space $\Bbb{P}^m$ and $(\Bbb{P}^1)^m$ are among the examples of such manifolds. 

For $1\leq k\leq m,$ we let $\Ck$ denote the set of all positive closed bidegree $(k,k)$ currents on $X$ which are cohomologous to $\omega^k$. This is a compact convex set. For a current $T\in \Ck,$ we denote its action on a test form $\Phi$ by $\la T,\Phi\ra.$ For a smooth $(p,q)$ form $\Phi$ denote by $\|\Phi\|_{\mathscr{C}^{\alpha}}$ the sum of $\mathscr{C}^{\alpha}$-norms of the coefficients in a fixed atlas. Following \cite{DS11}, for $\alpha>0$ we define a distance function on $\Ck$ by
$$dist_{\alpha}(R,R'):=\sup_{\|\Phi\|_{\mathscr{C}^{\alpha}}\leq 1}|\la R-R',\Phi\ra|$$
where $\Phi$ is a smooth bidegree $(m-k,m-k)$ form on $X.$ It follows from interpolation theory between Banach spaces \cite{Triebel} that
$$dist_{\beta}\leq dist_{\alpha}\leq C_{\alpha\beta}[dist_{\beta}]^{\frac{\alpha}{\beta}}$$
for $0<\alpha\leq \beta<\infty$ (see \cite[Lem. 2.1.2]{DS11} for the proof).  Moreover, for $\alpha\geq1$ 
$$dist_{\alpha}(\delta_a,\delta_b)\simeq\|a-b\|$$
where $\delta_a$ denotes the Dirac mass at $a$ and $\|a-b\|$ denotes the distance on $X$ induced by the K\"ahler form $\omega.$ We also remark that for $\alpha>0$ topology induced by $dist_{\alpha}$ coincides with the weak topology on $\Ck$ (cf. \cite[Prop. 2.1.4]{DS11}). In particular, $\Ck$ is a compact separable metric space. 
\\ \indent
 Let $T\in \Ck$ with $1\leq k\leq m$ then  by $dd^c$-Lemma \cite{GH} there exists a real $(k-1,k-1)$ current $U,$ called a \textit{quasi-potential} of $T$ which satisfies the equation
 \begin{equation}\label{qpotential}
 T=\omega^k+dd^cU
 \end{equation}
 where $d=\partial+\overline{\partial}$ and $d^c:=\frac{i}{2\pi}(\overline{\partial}-\partial)$ so that $dd^c=\frac{i}{\pi}\partial\overline{\partial}$.
  In particular, if $k=1$ a quasi-potential is nothing but a qpsh function. Note that two qpsh functions satisfying (\ref{qpotential}) differ by a constant. When $k>1$ the quasi-potentials differ by $dd^c$-closed currents. 
  For a real current $U$ and $0<r\leq \infty,$ we denote the sum of $\mathscr{L}^r$ norms of its coefficients for a fixed atlas by $\|U\|_{\mathscr{L}^r}.$ The quantity $\la U,\omega^{m-k+1}\ra$ is called the \textit{mean} of $U.$ The following result provides solutions to (\ref{qpotential}) with quantitative estimates.
  \begin{thm}\cite{DS11}\label{DS}
  Let $X$ be a compact K\"ahler homogenous manifold and $T\in \Ck$ then there exists a negative quasi-potential $U$ of $T$ which depends linearly on $T$ such that 
 for every $1\leq r\leq\frac{m}{m-1}$ and $1\leq s<\frac{2m}{2m-1}$ we have
  $$\|U\|_{\mathscr{L}^r}\leq c_r\ \ \text{and}\ \ \|dU\|_{\mathscr{L}^s}\leq c_s$$ 
  where $c_r,c_s$ are positive constants independent of $T.$ Moreover, $U$ depends continuously on $T$ with respect to the $\mathscr{L}^r$ topology on $U$ and weak topology on $T.$ 
  % the mean of $U$ satisfies  $$| \la U,\ommk \ra|\leq C$$  
 % where $C>0$ independent of $T\in \Ck.$
  \end{thm}
  %\begin{thm}\cite{DS11}\label{DSq}
  %Let $T\in \Ck.$ Then there exists a negative quasi-potential $U$ of $T$ which depends linearly on $T$ such that for every $1\leq r\leq\frac{k}{k-1}$ and $1\leq s<\frac{2k}{2k-1}$ we have
  %$$\|U\|_{\mathscr{L}^r}\leq c_r\ \ \text{and}\ \ \|dU\|_{\mathscr{L}^s}\leq c_s$$ 
  %where $c_r,c_s$ are positive constants independent of $T.$ 
  %\end{thm}
%Note that an immediate corollary of Theorem \ref{DSq} is that the mean of $U$ is bounded by a constant which is independent of $T\in\Ck.$
The quasi-potential $U$ is obtained in \cite{DS11} by using a kernel which solves $dd^c$-equation for the diagonal $\Delta$ of $X\times X$ (see also \cite{GiS, BGS}). More precisely, the current of integration $[\Delta]$ along the diagonal $\Delta\subset X\times X$ defines a positive closed $(m,m)$ current. It follows from K\"{u}nneth formula that $[\Delta]$ is cohomologous to a smooth real closed $(m,m)$ form $\Omega$ which is a linear combination of smooth forms of type $\beta(z)\wedge \beta'(\zeta)$ where $\beta$ and $\beta'$ are closed real forms on $X$ of bidegree $(r,m-r)$ and $(m-r,r)$ respectively (cf. \cite[\S 2.1]{DS10}). Then by \cite[Proposition 2.3.2]{DS11} there exists a negative $(m-1,m-1)$ form $K$ on $X\times X,$ smooth away from $\Delta$ such that $$dd^cK=[\Delta]-\Omega$$ satisfying    
\begin{equation}\label{DSes}
\|K(\cdot)\|_{\infty} \lesssim -\text{dist}(\cdot,\Delta)^{2(1-k)}\log \text{dist}(\cdot,\Delta)\ \ \text{and}\ \ \|\nabla K(\cdot)\|_{\infty} \lesssim \text{dist}(\cdot,\Delta)^{1-2k}
\end{equation}
where $\|\nabla K\|_{\infty}$ denotes the sum $\sum_j|\nabla K_j|$ and $K_j$'s are the coefficients of $K$ for a fixed atlas of $X\times X.$ 
 This implies that for $T\in\Ck,$ the $(k-1,k-1)$ current 
$$U(z):=\int_{z\not=\zeta}(T(\zeta)-\omega^k(\zeta))\wedge K(z,\zeta)$$
is well defined (cf. \cite[Theorem 2.3.1]{DS11}). Moreover, $$dd^cU=T-\omega^k.$$ Indeed, let $\pi_i:X\times X\to X$ denote the projection on the $i$th coordinate with $i=1,2.$ Note that $$U=(\pi_1)_*(\pi_2^*(T-\omega^k)\wedge K)$$ and since $T-\omega^k$ is closed 
\begin{eqnarray*}
dd^cU &=& (\pi_1)_*(\pi_2^*(T-\omega^k)\wedge dd^c K)\\&=& (\pi_1)_*(\pi_2^*(T-\omega^k)\wedge [\Delta])-(\pi_1)_*(\pi_2^*(T-\omega^k)\wedge \Omega)\\
&=& T-\omega^k
\end{eqnarray*}
where the last equality follows from observing that the cohomology class $\{T-\omega^k\}=0$ in $H^{k,k}(X,\Bbb{R})$ and $\Omega$ is a linear combination of smooth forms of type $\beta\wedge \beta'$ with $\beta$ and $\beta'$ are closed.
\\ \indent
Super-potentials of positive closed currents were introduced by Dinh and Sibony \cite{DS11} in the setting of complex projective space $\Bbb{P}^m$ (see also \cite{DS10}). The approach of \cite{DS11} can be easily extended to compact K\"ahler homogeneous manifolds. If $T$ is a smooth form in $\Ck,$ \textit{super-potential} of $T$ of mean $c$ is  defined by
$$\U_T:\Cmk \to \Bbb{R}\cup\{-\infty\}$$  
\begin{equation}\label{superr}
\U_T(R)=\la T,U_R\ra
\end{equation}
where $U_R$ is a quasi-potential of $R$ of mean $c.$ Then it follows that (see \cite[Lemma 3.1.1]{DS11})
\begin{equation}\label{sym}
\U_T(R)=\la U_T,R\ra
\end{equation}
where $U_T$ is a quasi-potential of $T$ of mean $c.$ In particular, the definition of $\U_T$ in (\ref{superr}) is independent of the choice of $U_R$ of mean $c$. Note that super-potential of $T$ of mean $c'$ is given by $\U_T+c'-c.$ More generally, for an arbitrary current $T\in\Ck$ super-potential of $T$ is defined by $\U_T(R)$ on smooth forms $R\in \Cmk$ as in (\ref{superr}) where $U_R$ is smooth. Then the definition of super-potential can be extended in a unique way to an affine usc function on $\Cmk$ with values in $\Bbb{R}\cup \{-\infty\}$ by approximation (see \cite[Proposition 3.1.6]{DS11} and \cite[Corollary 3.1.7]{DS11}). Namely, $$\U_T:\Cmk \to \Bbb{R}\cup \{-\infty\}$$
$$\U_T(R)=\limsup_{R'\to R}\U_T(R')$$ where $R'\in \Cmk$ is smooth.
 %In the sequel, for $T\in \Ck$ we denote $\|\U_T\|_{\infty}=\sup_{R\in\Ck}|\U_T(R)|.$ Note that the later $\sup$ is finite if $T\in \Ck$ is smooth. 
\begin{rem}\label{L1} 
It follows from Theorem \ref{DS} that for each $T\in \Ck$ there exists a negative super potential $\U_T$ of $T$ such that its mean satisfies $$|\U_T(\ommk)|\leq C$$ where $C>0$ is independent of $T\in \Ck.$
\end{rem}

Another feature of super-potentials is that for each $1\leq k\leq m,$ one can define a function $$\U_k:\Ck\times \Cmk \to \Bbb{R}$$
$$\U_k(T,R):=\U_T(R)=\U_R(T)$$ where $\U_T$ and $\U_R$ are super-potentials of $T$ and $R$ of the same mean. Moreover, $\U_k$ is u.s.c. (cf.  \cite[Lemma 4.1.1]{DS11}):
\begin{lem}\label{Climsup}
Let $T_n\in \Ck, R_n\in\Cmk$ be sequences of positive closed currents such that $T_n\to T$ and $R_n\to R$ in the sense of currents as $n\to \infty.$ Then $$\limsup_{n\to \infty}\U_{T_n}(R_n)\leq \U_T(R).$$ 
\end{lem}

Next result indicates that super-potentials determine the currents:
\begin{prop}\label{determine}
Let $T,T'$ be currents in $\Ck$ with super-potentials $\U_T,\U_{T'}$ of mean $c$. If $\U_T=\U_{T'}$ on smooth forms in $\Cmk$ then $T=T'.$
\end{prop}
\begin{proof}
Let $\Phi$ be a smooth $(m-k,m-k)$ form. Then there exists $C>0$ such that $C\omega^{m-k+1}+dd^c\Phi\geq0.$ Thus, $$\U_T(C\omega^{m-k+1}+dd^c\Phi)=\U_{T'}(C\omega^{m-k+1}+dd^c\Phi)$$ which implies that 
$$\la T,\Phi\ra=\la T',\Phi\ra.$$ 
\end{proof}
\subsection{Currents with continuous super-potentials}
 In this section we consider currents $T\in\Ck$ with continuous super-potentials.\\
 {\bf{The space $DSH^{m-k}(X)$:}} Following \cite{DS3,DS11}, a real $(m-k,m-k)$ current $\Phi$ of finite mass is called dsh if there exists positive closed currents $R^{\pm}$ of bidegree $(m-k+1,m-k+1)$ such that $dd^c\Phi=R^+-R^-.$ Then one can define $$\|\Phi\|_{DSH}:=\|\Phi\|+\min\|R^{\pm}\|$$ where $\|R^{\pm}\|:=|\int_X R^{\pm}\wedge \omega^{k-1}|.$ Note that since $R^+$ and $R^-$ are cohomologous we have $\|R^+\|=\|R^-\|.$ We consider the space $DSH^{k-p}(X)$ with the weak topology: we say that $\Phi_n$ converges to $\Phi$ if 
 \begin{itemize}
 \item $\Phi_n\to \Phi$ in the sense of currents
 \item $\|\Phi_n\|_{DSH}$ is bounded
 \end{itemize}
  It follows from Theorem \ref{DS} that if $R_n \to R$ weakly in $\Ck$ then there exists negative quasi-potentials $U_n,U$ of $R_n,R$ such that $U_n$ converges to $U$ in $DSH^{k-1}(X).$ A positive closed current $T\in \Ck$ is called \textit{PC} if $T$ can be extended to a linear continuous form on $DSH^{m-k}(X).$ We denote the value of the extension by $\la T,\Phi\ra.$ Since smooth forms are dense in $DSH^{m-k}(X)$ the extension is unique. The following result is a consequence of Theorem \ref{DS} (see \cite[Proposition 3.3.1]{DS11}).
  
  \begin{prop}
  A positive closed current is PC if and only if $T$ has continuous super-potentials.
  \end{prop} 
The next result is also adapted from \cite{DS11} that relates continuity of quasi-potentials to that of super-potentials of a positive closed current. We provide a proof for convenience of the reader.     
\begin{prop} \label{cont} Let $T\in\mathscr{C}_1$
\begin{itemize}
\item[(1)]  $T$ is PC if and only if $T$ has continuous quasi-potentials.

\item[(2)] Let $R\in \Ck$ be a PC current and $T=\omega+dd^c u$ for some continuous qpsh function $u.$ Then $T\wedge R$ is a PC current. In particular, $$T^k:=T \wedge \dots \wedge T$$ is PC  for $1\leq k\leq m.$ 
\end{itemize}
\end{prop}
\begin{proof}
\begin{itemize}
\item[(1)] Let $T=\omega+dd^c u$ where $u$ is continuous and $\Phi$ be a current in $DSH^{m-1}(X).$ We write $dd^c\Phi=\nu^+-\nu^-$ for some measures $\nu^{\pm}\in \mathscr{C}_m.$ If $\Phi$ and $\nu^{\pm}$ are smooth then 
\begin{eqnarray*}
\la T,\Phi\ra &=& \la \omega,\Phi\ra+\la u,dd^c\Phi\ra\\
&=& \la \omega,\Phi\ra+\la u,\nu^+\ra-\la u,\nu^-\ra
\end{eqnarray*}
since right hand side is well-defined and depends continuously on $\Phi\in DSH^{m-1}(X)$ we conclude that $T$ extends to a continuous linear form on $DSH^{m-1}(X).$\\
Conversely, for $x\in X$ let $\Phi_x$ be a $(m-1,m-1)$ current satisfying $dd^c\Phi_x=\delta_x-\omega^m.$ Then using a regularization of $\Phi_x$ we see that $$\la T,\Phi_x\ra= \la \omega,\Phi_x\ra-\la u,\omega^m\ra+u(x)$$ since $\Phi_x$ and $\la T,\Phi_x \ra$ depend continuously on $x$ we conclude that  $u$ is continuous on $X.$
\item[(2)] Since $T$ has locally bounded potentials the current $T\wedge R$ is well-defined \cite{BT2}. Now, if $\Phi\in DSH^{m-k-1}(X)$ is a smooth form then one can define 
$$\la T\wedge R,\Phi\ra:=\la R,\omega\wedge \Phi\ra +\la R,udd^c\Phi\ra$$
When $\Phi$ is not smooth the right hand side is still well-defined and depends continuously on $\Phi.$ Indeed, since $R$ is PC, for every positive closed $(m-k,m-k)$ current $S$ the measure $R\wedge S$ is well-defined and depends continuously on $S.$ This is because for every smooth function $\varphi$ the current $\varphi S$ is DSH. Thus, we can define $$\la R\wedge S,\varphi\ra:=\la R,\varphi S\ra.$$
Hence, $T\wedge R$ extends to a continuous linear form on $DSH^{m-k-1}(X).$  
\end{itemize}
\end{proof}
\subsubsection{Super-potentials of intersection products}
Let $T_1$ and $T_2$ be two positive closed current of bidegree $(1,1)$ and $(k,k)$ respectively and assume that $1\leq k\leq m-1.$ We also let $T_1=\omega+dd^c\varphi$ where $\varphi$ is a qpsh function and $\U_{T_1}$ denote the super-potential of $T_1$ of mean zero. Recall that the wedge product $T_1\wedge T_2$ is well-defined in the sense of currents if and only if  $\varphi\in L^1(T_2\wedge\omega^{m-k})$ (see \cite[Chapter 1]{DemBook}). In this case, by \cite[\S 4]{DS11} the super-potential of $T_1\wedge T_2$ of mean zero is given by
$$\U_{T_1\wedge T_2}:\mathscr{C}_{m-k}\to \Bbb{R}$$
\begin{equation}\label{wedge}
\U_{T_1\wedge T_2}(R)=\la T_1\wedge T_2,U_R\ra=\la T_2, \omega\wedge U_R\ra+\U_{T_1}(T_2\wedge R)-\U_{T_1}(T_2\wedge \omega^{m-k})
\end{equation}
whenever $R$ is a smooth form and $U_R$ is a smooth quasi-potential of $R$ of mean zero.

\subsection{Holomorphic sections}
Let $X$ be a projective manifold of complex dimension $m$. Given a holomorphic line bundle $\pi:L\to X$ we can find an open cover $\{U_{\alpha}\}$ of $X$ and biholomorphisms $\varphi_{\alpha}:\pi^{-1}(U_{\alpha}) \to U_{\alpha}\times \Bbb{C}$, trivializations of $\pi^{-1}(U_{\alpha}).$ Then the line bundle $L$ is uniquely determined (up to isomorphism) by the transition functions $g_{\alpha\beta}:U_{\alpha} \cap U_{\beta} \to \Bbb{C}$ defined by $g_{\alpha\beta}(z)=(\varphi_{\alpha}\circ\varphi_{\beta}^{-1})_{|_{\pi^{-1}(z)}}.$ The functions $g_{\alpha\beta}$ are non-vanishing holomorphic functions on $U_{\alpha\beta}:=U_{\alpha}\cap U_{\beta}$ satisfying
 \begin{equation*}
 \left\{
\begin{array}{rl}
g_{\alpha\beta}.g_{\beta\alpha}=1\\
g_{\alpha\beta}.g_{\beta\gamma}.g_{\gamma\alpha}=1
\end{array} \right.
\end{equation*}
\indent  Recall that a singular metric $h$ is given by a collection $\{e^{-\psi_{\alpha}}\}$ of functions $\psi_{\alpha}\in L^1(U_{\alpha})$ which are called \textit{weight functions} with respect to the trivializations $\varphi_{\alpha}$  satisfying the compatibility conditions $\psi_{\alpha}=\psi_{\beta}+\log|g_{\alpha\beta}|$ on $U_{\alpha\beta}$. We say that the metric is positively curved if $\psi_{\alpha} \in Psh(U_{\alpha})$ and the metric is called smooth if $\psi_{\alpha}\in C^{\infty}(U_{\alpha})$ for every $\alpha.$ %Recall that $L\to X$ is called pseudo-effective (psef for short) if $L$ admits a singular positively curved metric. Moreover, we say that $L\to X$ is semi-ample if the base locus of $L^{\otimes m}$ is empty for some $n\in \Bbb{N}.$ If $X$ is complex homogenous manifold then it follows from \cite{H94} that every positive closed $(1,1)$ current $T$ can be approximated by smooth semi-positive forms $\theta_{\epsilon}$ which are cohomologous to $T.$ In particular, this implies that every psef line bundle $L\to X$ on a projective homogenous manifold is semi-ample. 
In the sequel, we assume that $L$ admits a smooth positively curved metric $h=\{e^{-\psi_{\alpha}}\}.$ Then its curvature form is locally defined by $$\omega:=dd^c\psi_{\alpha}$$ which is a globally well-defined smooth closed $(1,1)$ positive form representing the class $c_1(L)$ where $c_1(L)$ is the image of the  Chern class of $L$ under the mapping $i:H^2(X,\Bbb{Z})\to \cohomology$ induced by the inclusion $i:\Bbb{Z}\to\Bbb{R}$ and $dV:=\omega^m$ induces a volume form on $X.$\\ \indent 
 A \textit{global holomorphic section} $s=\{s_{\alpha}\}$ of $L^{\otimes n}$ is a collection of holomorphic functions satisfying the compatibility conditions $s_{\alpha}=s_{\beta}.g^n_{\alpha\beta}$ on $U_{\alpha\beta}$. We denote the set of all global holomorphic sections by $H^0(X,L^{\otimes n}).$ 
 % The metric $h$  induces an inner product on $H^0(X,L^{\otimes n})$ which is defined by $$\langle s_1(x),s_2(x)\rangle_{h_n}:=s_1(x)\overline{s_2(x)}e^{-2n\psi_{\alpha}(x)}$$ for $s_1,s_2\in H^0(X,L^{\otimes n})$ and $x\in U_{\alpha}$. 
For $s\in H^0(X,L^{\otimes n})$ we also set  $$|| s(x)||_{h_n}:=|s_{\alpha}(x)|e^{-n\psi_{\alpha}(x)}$$ on $U_{\alpha}.$ By compatibility conditions this definition is independent of $\alpha$.\\ \indent
  For $s\in H^0(X,L^{\otimes n})$ we let $[Z_s]$ denote the current of integration along the zero divisor of $s.$ Then by Poincar\'e-Lelong formula, locally we can write $$[Z_s]=dd^c\log|s_{\alpha}|$$ on $U_{\alpha}$ hence, $$[Z_s]=n\omega+dd^c\log||s||_{h_n}$$ on $X$ where the equality follows from compatibility conditions. Thus, we conclude that $\widetilde{Z}_s:=\frac{1}{n}[Z_s]$ is a positive closed $(1,1)$ current representing the class  $c_1(L)$ that is $\Z_{s}\in\mathscr{C}_1$. \\ \indent
  On the other hand, if $X$ is complex homogeneous then for each $s\in H^0(X,L^{\otimes n})$ we let 
  \begin{equation}\label{choice}
  \varphi_s(z):=\int_{z\not=\zeta}(\Z_{s}(\zeta)-\omega(\zeta))\wedge K(z,\zeta).
  \end{equation} Then by Theorem \ref{DS} we have 
$$ \Z_s=\omega+dd^c \varphi_s$$
where $\varphi_s\leq 0$ and 
  \begin{equation*}
  |\int_X\varphi_s dV|\leq C
   \end{equation*}
where $C>0$ independent of $s\in H^0(X,L^{\otimes n})$ and $n\in\Bbb{N}.$ Note that since $\varphi_s\leq 0$ the $L^1(\omega^m)$-norms of $\varphi_s$ are uniformly bounded; thus, by Hartog's lemma \cite{DemBook} the set $\{\varphi_s: s\in\cup_{n=1}^{\infty}\mathcal{S}_{n}\}$ is pre-compact in $L^1(\omega^m).$\\ \indent
In the sequel, we denote the super-potential of $\Z_s$ by $$\U_{\Z_s}:\mathscr{C}_m\to [-\infty,0]$$ 
\begin{equation}\label{super}
\U_{\Z_s}(\nu)=\int_X\varphi_s d\nu
\end{equation}
We remark that above definition of $\U_{\Z_{s}}$ depends on the choice of the Kernel $K(z,\zeta)$ which is fixed throughout this paper. On the other hand, for smooth $\nu\in\mathscr{C}_m$ and $U_{\nu}$ is a smooth quasi-potential of $\nu$ of mean equal to $\int_X\varphi_sdV$ then by (\ref{sym}) we have 
 \begin{equation}\label{number}
 \U_{\Z_s}(\nu)=\la \Z_s,U_{\nu}\ra.
 \end{equation}
\subsection{Global extremal function}\label{EF}
A subset $K$ is said to be (locally) pluripolar if $K\cap U_{\alpha}\subset \{u=-\infty\}$ for some $u\in Psh(U_{\alpha}).$ We say that $K$ is globally \textit{pluripolar} if $K\subset\{\varphi=-\infty\}$ for some $\varphi\in PSH(X,\omega).$ It follows from \cite[Theorem 7.2]{GZ} that locally pluripolar sets are $PSH(X,\omega)$- pluripolar. Throughout this paper we assume that $K\subset X$ is a non-pluripolar compact set. We also let $q:K\to \Bbb{R}$ be a continuous function. %By a \textit{weight function} $q$ on $K$ we mean a weight of a continuous Hermitian metric $e^{-q}$ on the restriction of the positive line bundle $L_{|_K}.$ Note that by Tietze extension theorem $q$ extends to a continuous weight on $L$ over $X.$ 
Following \cite{GZ} we define the \textit{weighted global extremal function} $V^*_{K,q}$ as usc regularization of  
$$V_{K,q}:=\sup\{\varphi\in PSH(X,\omega)|\ \varphi(x)\leq q(x)\ \text{for}\ x\in K\}.$$
 It follows from \cite[\S 5]{GZ} that $V_{K,q}^*\in PSH(X,\omega)$, that is, $$T_{K,q}:=\omega+dd^c V_{K,q}^*$$ defines a positive closed $(1,1)$ current. Throughout this paper we assume that $V_{K,q}$ is continuous so that $V_{K,q}=V_{K,q}^*.$ It is well know that if $K$ is locally regular compact set in $\Bbb{C}^m$  then $V_{K,q}$ is continuous \cite[Proposition 2.16]{Siciak}. The argument in \cite{Siciak} ((see also \cite{BBN}) can be adapted to our setting. Moreover, we have the following version of Siciak-Zaharjuta theorem:\begin{thm}\label{ZS}
Let $(K,q)$ be as above then
$$V_{K,q}=\sup\{\frac{1}{n}\log \|s(x)\|_{h_n}: s\in \cup_{n=1}^{\infty}H^0(X,L^{\otimes n})\ \text{and}\ \max_{x\in K}(\|s(x)\|_{h_n}e^{-nq(x)})\leq 1\}$$
\end{thm} 
Theorem \ref{ZS} was proved in \cite[\S 6]{GZ} for unweighted case (i.e. $q\equiv0$) and the argument in \cite{GZ} carries over to the setting of a continuous weight function $q.$ 

\subsection{Bernstein-Markov inequality}\label{BM}
Let $K\subset X$ be a non-pluripolar compact set, $q$ be a continuous weight on $K$ and $\tau$ be a measure supported on $K.$ We say that the triple $(K,q,\tau)$ satisfies the weighted \textit{Bernstein-Markov inequality}\label{prob} if 
$$\max_{x\in K}(\|s(x)\|_{h_n}e^{-nq(x)})\leq M_n(\int_K\|s(z)\|^2_{h_n}e^{-2nq(z)}d\tau)^{\frac12}$$
for all $s\in H^0(X,L^{\otimes n})$ and $\displaystyle\limsup_{n\to \infty}(M_n)^{\frac1n}=1$.\\ \indent
It follows that such a triple induces a weighted $L^2$-norm on $H^0(X,L^{\otimes n})$ defined by
\begin{equation}\label{inner}
\|s\|_{L^2_{q,\tau}}:=\big(\int_K\|s(x)\|_{h_n}^2e^{-2nq(x)}d\tau(x)\big)^{\frac12}
%\langle s_1,s_2\rangle_{q,\tau}:=\int_K\langle s_1(x),s_2(x)\rangle_{h_n}e^{-2nq(x)}d\tau(x)
\end{equation}
and we fix an orthonormal basis $\{S_j^n\}_{j=1}^{d_n}$ for $H^0(X,L^{\otimes n})$ with respect to the inner product defined by (\ref{inner}). Then a section $s_n\in H^0(X,L^{\otimes n})$ can be written uniquely as 
 $$s_n(x)=\sum_{j=1}^{d_n}a^{(n)}_jS_j^n(x).$$
where $d_n:=\dim H^0(X,L^{\otimes n}).$  Recall that since $L$ is positive there exists a constant $C>0$ such that $C^{-1}n^m\leq d_n\leq Cn^m$ for $n\in\Bbb{N}.$\\ \indent
 In the sequel, we assume that $a^{(n)}_j$ are i.i.d. complex (or real) valued random variables whose distribution $\prob=\phi(z)d\lambda$ is absolutely continuous with respect to the Lebesgue measure $\lambda$ on $\C$ and its density function satisfies  
$$0\leq \phi(z)\leq C\ \text{for all}\ z\in\C$$
$$\prob\{z\in\Bbb{C}: |z|>R\}\leq\frac{C}{(\log R)^{\rho}} \ \ \text{for sufficiently large}\ R>0$$ 
where $\rho>\dim X+1.$\\ \indent
 In what follows, we consider the coefficients $a^{(n)}:=\{a^{(n)}_j\}_j$ as points in $\C^{d_n}$ and $(\C^{d_n},\probn)$ as a probability space where $\probn$ is the $d_n$-fold product measure induced by $\prob$. We also denote $$\|a^{(n)}\|:=(\sum_{j=1}^{d_n} |a_j^{(n)}|^2)^{\frac12}$$ is the $\ell^2$ norm on $\C^{d_n}.$ Finally, we define $\mathcal{C}:=\prod_{n=1}^{\infty}\C^{d_n}$ endowed with the product measure $\p:=\prod_{n=1}^{\infty}\probn$ and consider the probability space $(\mathcal{C},\p).$ The following lemma will be useful in the sequel:
\begin{lem}\label{limsup}
For $\p$-a.e.  $\{a^{(n)}\}_{n\geq1}\in \mathcal{C}$ 
$$\lim_{n\to \infty}\frac1n\log\|a^{(n)}\|=0.$$ 
\end{lem} 
\begin{proof}
We fix $\epsilon>0$ such that $\rho(1-\epsilon)>m+1.$ First, we show that with probability one,  $$\|a^{(n)}\|\leq d_ne^{n^{1-\epsilon}}$$ for sufficiently large $n.$ Indeed,
$$\prob\{a_j^{(n)}\in\C: |a_j^{(n)}|>e^{n^{1-\epsilon}}\}\leq \frac{C}{n^{(1-\epsilon)\rho}}$$ which implies that 
$$\probn\{a^{(n)}\in\C^{d_n}: \|a^{(n)}\|>d_ne^{n^{1-\epsilon}}\}\leq \frac{Cd_n}{n^{(1-\epsilon)\rho}}$$ where the later defines a summable sequence. Thus, the claim follows from Borel-Cantelli lemma. %This implies that with probability one $$\|a^{(n)}\|\leq d_n e^{n^{1-\epsilon}}$$ 
Since $d_n=O(n^m)$ we conclude that
$$\limsup_{n\to \infty}\frac1n\log\|a^{(n)}\|\leq 0$$
with probability one. \\ \indent
On the other hand, since $|\phi(z)|\leq C$ by independence of $a_j^{(n)}$'s
\begin{eqnarray*}
\probn\{a^{(n)}\in\C^{d_n}:\|a^{(n)}\|<\frac1n\}\leq\prob\{z\in\Bbb{C}:|z|<\frac1n\} &\leq& C \lambda\{z\in\C:|z|<\frac1n\}\\
& = & \frac{C\pi}{n^2}
\end{eqnarray*}
thus, again, by using Borel-Cantelli lemma we obtain $$\|a^{(n)}\|\geq\frac{1}{n}$$ with probability one. Hence, $$\liminf_{n\to\infty}\frac1n\log\|a^{(n)}\|\geq 0.$$
\end{proof}

\subsection{Bergman kernel asymptotics}
A triple $(K,\tau,q)$ satisfying the Bernstein-Markov inequality induces an inner product (\ref{inner}) on the space of global sections with values in $L^{\otimes n}.$ Recall that \textit{Bergman kernel} for the Hilbert space $H^0(X,L^{\otimes n})$ is the integral kernel of the orthogonal projection from $L^2$-space of global sections with values in $L^{\otimes n}$ onto $H^0(X,L^{\otimes n}).$ It is well-known that it can be represented as a holomorphic section 
$$S_n(x,y)=\sum_{j=1}^{d_n}S^n_j(x)\otimes \overline{S_j^n(y)}$$
of the line bundle $L^{\otimes n}\boxtimes\overline{L^{\otimes n}}$ over $X\times X.$ The point-wise norm of restriction of $S_n(x,y)$ to the diagonal is given by 
$$\|S_n(x,x)\|_{h_n}=\sum_{j=1}^{d_n}\|S^n_j(x)\|^2_{h_n}.$$ 
%Following \cite{Berman} we refer to $\|S_n(x,x)\|_{h_n}$ as the Bergman function. 
The following result is well-known in the $\Bbb{C}^m$ setting \cite[Lemma 3.2]{BloomS} and the argument in \cite{BloomS} can be adapted to our setting. (cf. \cite[Theorem 1.5]{Berman} for the case $K=X$).  
\begin{prop} \label{bergman}
Let $S_n(z,z)$ denote the Bergman kernel on the diagonal then
$$\lim_{n\to \infty}\frac{1}{2n}\log \|S_n(x,x)\|_{h_n}=V_{K,q}(x)$$
for every $x\in X.$ Moreover, if $V_{K,q}$ is continuous then the convergence is uniform on $X$.
\end{prop}
The next observation will be useful in the sequel.

\begin{prop}\label{BL}
For $\mu$-a.e. $\{s_n\}\in\mathcal{S}_{\infty}$ $$\limsup_{n\to\infty} \frac1n\log \|s_n(x)\|_{h_n}\leq V_{K,q}(x)$$ for every $x\in X.$
\end{prop}
\begin{proof}
%Indeed, by Bernstein-Markov inequality for every $x\in K$
 %$$\|s_n(x)\|_{h_n}e^{-nq(x)}\leq M_n \|s_n\|_{L^2_{\tau,q}}$$
 %this implies that
 %$$\frac1n\log\|s_n(x)\|_{h_n}-q(x)\leq \frac1n\log M_n+\frac{1}{2n}\log\|s_n\|_{L^2_{\tau,q}}$$
%$$\varphi(x)-q(x)\leq 0$$
%for $x\in K.$ This completes the proof.

We write  $s_n=\sum_{j=1}^{d_n}a^{(n)}_jS_j^n$ then by Cauchy-Schwarz inequality
$$\frac1n \log\|s_n(x)\|_{h_n}\leq \frac{1}{2n}\log\sum_{j=1}^{d_n}|a^{(n)}_j|^2 + \frac{1}{2n}\log\sum_{j=1}^{d_n}\|S_j^n(x)\|_{h_n}^2$$ thus the assertion follows from Lemma \ref{limsup} and Proposition \ref{bergman}.
\end{proof}

\section{Expected distribution of zeros}
In this section we prove the first part of Theorem \ref{main}. For this, we assume that $X$ is merely projective and $L\to X$ is a positive holomorphic line bundle. The following lemma is an improved version of  \cite[Proposition 7.2]{BloomL}:% but its proof uses the same argument as in \cite{BloomL}:

\begin{lem}\label{help}
Let $u\in \C^{d_n}$ be a unit vector then
$$\int_{\C^{d_n}}|\log|\langle a,u\rangle|| d\probn(a)=O(n^{1-\epsilon})$$ for some small $\epsilon>0.$
\end{lem}
\begin{proof}
We fix $\epsilon>0$ such that $(\rho-1)(1-\epsilon)\geq m.$ First, we prove that
\begin{equation}\label{es1}
\int_{\{\log|\langle a,u\rangle|>m n^{1-\epsilon}\}}\log|\langle a, u\rangle| d\probn(a)\leq C_m
\end{equation}
 where $C_m>0$ is a constant which depends only on $m.$ %Indeed, by symmetry we can get the same estimate on $\{a\in \C^{d_n}:\log|\langle a, u\rangle|<-mn^{1-\epsilon}\}$ and since $\mu_n$ is a probability measure this would complete the proof. Next, we prove the estimate (\ref{es1}). 
 Note that for $k\in\Bbb{N}$
 $$\{a\in\C^{d_n}:\log|\langle a,u\rangle|>kn^{1-\epsilon}\}\subset \{a\in\C^{d_n}:\|a\|>e^{kn^{1-\epsilon}}\}\subset \bigcup_{j=1}^{d_n}\{a_j\in \Bbb{C}: |a_j|\geq \frac{e^{kn^{1-\epsilon}}}{\sqrt{d_n}}\}$$
 Hence, by (\ref{h2}) for sufficiently large $n$ we obtain 
 \begin{equation}\label{rk}r_k(n):=\probn\{a\in\C^{d_n}:\log|\langle a,u\rangle| >kn^{1-\epsilon}\}\leq \frac{Cn^m}{k^{\rho}n^{(1-\epsilon)\rho}}
 \end{equation}
 for every $k\geq m.$ Now, letting
 $$R_k(n):=\{a\in\C^{d_n}: kn^{1-\epsilon}<\log|\langle a,u\rangle| \leq(k+1)n^{1-\epsilon}\}$$
 then $$\probn(R_k(n))=r_k(n)-r_{k+1}(n)$$ and 
we see that
\begin{eqnarray*}
\int_{\{\log|\langle a,u\rangle|>m n^{1-\epsilon}\}} \log|\langle a,u\rangle| d\probn(a) &\leq& \sum_{k=m}^{\infty}(k+1) n^{1-\epsilon}(r_k(n)-r_{k+1}(n))\\
&\leq& n^{1-\epsilon}((m+1)r_m(n)+\sum_{k=m+1}^{\infty}r_k(n))\\
&\leq& n^{1-\epsilon}((m+1)\frac{Cn^m}{m^{\rho}n^{\rho(1-\epsilon)}}+\frac{Cn^m}{n^{\rho(1-\epsilon)}}\sum_{k=m+1}^{\infty}k^{-\rho})\\
&\leq& \frac{Cn^m}{n^{(\rho-1)(1-\epsilon)}}(\frac{m+1}{m^{\rho}}+\sum_{k=m+1}^{\infty}k^{-\rho}). \\
\end{eqnarray*}
Since $(\rho-1)(1-\epsilon)\geq m$ the claim follows.\\ \indent
Next, we show that 
\begin{equation}\label{es2}
\int_{\{\log|\langle a,u\rangle|<-m n^{1-\epsilon}\}}|\log|\langle a, u\rangle|| d\probn(a)\leq C_m
\end{equation}
 where $C_m>0$ is a constant which depends only on $m.$ To this end we show that
 $$\l_k(n):=\probn\{a\in\C^{d_n}:|\la a,u\ra|\leq e^{-kn^{1-\epsilon}}\}\leq Cd_ne^{-2kn^{1-\epsilon}}.$$
where $C>0$ constant as in (\ref{h1}), in particular, independent of $n$. Indeed, since $u$ is a unit vector we may assume that $|u_1|\geq\frac{1}{\sqrt{d_n}}$ where $u=(u_1,u_2,\dots,u_{d_n}).$ Following \cite[Lemma 2.8]{BloomL} we apply the change of variables $\alpha_1=\sum_{i=1}^{d_n}a_iu_i, \alpha_2=a_2,\dots,\alpha_{d_n}=a_{d_n}$ and we obtain  
 \begin{eqnarray*}
\l_k &=& \int_{\C^{d_n-1}}\int_{|\alpha_1|\leq  e^{-kn^{1-\epsilon}}}\frac{1}{|u_1|^2}\phi(\frac{\alpha_1-\alpha_2u_2-\dots-\alpha_{d_n}u_{d_n}}{u_1})\phi(\alpha_2)\dots\phi(\alpha_{d_n})d\lambda(\alpha_1)\dots d\lambda(\alpha_{d_n}) \\
&\leq& Cd_ne^{-2kn^{1-\epsilon}}
 \end{eqnarray*}
 Next, for $k\geq m$ we let $$D_k(n):=\{a\in \C^{d_n}:e^{-(k+1)n^{1-\epsilon}}< |\la a,u\ra|\leq e^{-kn^{1-\epsilon}}\}$$ then 
 \begin{equation}\label{dk}
 \probn(D_k(n))\leq \probn\{a\in\C^{d_n}:|\la a,u\ra|\leq e^{-kn^{1-\epsilon}}\}\leq Cn^me^{-2kn^{1-\epsilon}}
 \end{equation}
and
\begin{eqnarray*}
\int_{\{\log|\langle a,u\rangle|<-m n^{1-\epsilon}\}}|\log|\langle a, u\rangle|| d\probn(a) &=& \sum_{k=m}^{\infty}\int_{D_k}|\log |\la a,u\ra||d\probn \\
&\leq& \sum_{k=m}^{\infty}(k+1)n^{m+1-\epsilon}e^{-2kn^{1-\epsilon}}\\
&\leq& n^{m+1-\epsilon}\int_m^{\infty}(x+1)e^{-2xn^{1-\epsilon}}dx\\
&\leq& Ce^{-2mn^{1-\epsilon}}n^m
\end{eqnarray*}
where $C>0$ depends only on $m$ which proves (\ref{es2}). \\ \indent
Combining (\ref{es1}) and (\ref{es2}) we conclude that
 $$\int_{\{|\log|\langle a,u\rangle||>m n^{1-\epsilon}\}}|\log|\langle a, u\rangle| |d\probn(a)\leq C_m$$
since $\probn$ is a probability measure this finishes the proof.
\end{proof}

\begin{thm}\label{expected}
Let $X$ be a projective manifold, $L\to X$ be a positive holomorphic line bundle and $K\subset X$ is a locally regular compact set together with a continuous function $q:K\to \Bbb{R}$ then
$$\Bbb{E}[\widetilde{Z}_s] \to T_{K,q}$$ in the sense of currents as $n\to \infty.$
\end{thm}
\begin{proof}
We need to show that for every smooth $(m-1,m-1)$ form $\Phi$ on $X$ %such that $dd^c\Phi=R^+-R^-$ for some smooth $R^{\pm}\in \mathscr{C}_m$
$$\int_{\mathcal{S}_n}\la \Z_{s},\Phi\ra d\mu_n(s) \to\la T_{K,q},\Phi\ra$$
as $n\to \infty.$
 We may assume that $supp(\Phi)\subset U_{\alpha}$ for some $\alpha.$ The general case follows from covering the $supp(\Phi)$ by $U_{\alpha}$'s and using the compatibility conditions. Following \cite{SZ}, let $e_{\alpha}:U_{\alpha}\to L$ be a holomorphic frame for $L$ over $U_{\alpha}$ then for $s\in H^0(X,L^{\otimes n})$ we may write $$s=\sum_{j=1}^{d_n}a_jf_je_L^{\otimes n}$$ on $U_{\alpha}$ and denote $$\la a,f\ra:=\sum_{j=1}^{d_n}a_jf_j$$
 where $S^n_j=f_je^{\otimes n}_L$ and $f=(f_1,f_2,\dots,f_{d_n}).$ Then by Poincar\'e Lelong formula
 $$\Z_s=\frac1n dd^c\log|\la a,f\ra|$$
 on $U_{\alpha}.$ Evidently,
 $$\log|\la a,f\ra|=\log|\la a,u\ra|+\log|f|$$ where $f=|f|u$ and $|u|\equiv 1$ on $U_{\alpha}.$ 
 Then it follows from Lemma \ref{help} that
 \begin{equation}\label{bound} 
 \int_{X\times \C^{d_n}}|\log|\la a,u\ra|dd^c\Phi|d\probn(a)%=\int_{\C^{d_n}}|\log|\la a,u\ra||d\probn(a) \int_X|dd^c\Phi|
  \leq C_mn^{1-\epsilon}\|dd^c\Phi\|_{\infty} 
 \end{equation}
for large $n.$ Hence, by Fubini's Theorem we obtain
 \begin{eqnarray*}
\int_{\mathcal{S}_n}\la \Z_s,\Phi\ra d\mu_n(s) &=& \int_{\Bbb{C}^{d_n}}\int_X \frac1n dd^c\log|\la a,f\ra|\wedge \Phi\ d\probn(a)\\
&=& \int_{X} \int_{\Bbb{C}^{d_n}}\frac1n \log|\la a,u(x)\ra| d\probn(a)\ dd^c\Phi(x)+ \int_X\frac{1}{2n}\log\sum_{j=1}^{d_n}|f_j|^2  dd^c\Phi\\
&=& I_1(n)+I_2(n)
 \end{eqnarray*}
 %Since $dd^c\Phi=R^+-R^-$ for some positive measures $R^{\pm}$ by Lemma \ref{help} 
 Note that by (\ref{bound}) we have $I_1(n)\to 0$ as $n\to \infty.$ On the other hand, by Proposition \ref{bergman} we obtain $I_2(n) \to \la T_{K,q}, \Phi\ra$ as $n\to \infty.$ This completes the proof.
\end{proof}
 We remark that if $a_j^{(n)}$ are standard complex Gaussians then the integral $I_1(n)$ in the proof of Theorem \ref{expected} is equal to zero (see \cite[Lemma 3.1]{SZ}). In particular, the expected zero current is given by 
 \begin{equation}\label{exg}
 \Bbb{E}[\Z_{s_n}]=(\Phi_n^S)^*\omega_{FS}
 \end{equation}
  where
  $$\Phi_n^S:X\to \Bbb{P}^{d_n-1}$$
  $$x\to [S_1^n(x):\dots:S_{d_n}^n(x)]$$
  is the Kodaira map defined by the orthonormal basis $S=\{S_1^n,\dots,S_{d_n}^n\}$ and $\omega_{FS}$ is the Fubini-Study form on the complex projective space $\Bbb{P}^{d_n-1}$. In our setting this is no longer the case. In fact, it follows from Theorem \ref{expected} that $$\Bbb{E}[\Z_{s_n}]=(\Phi_n^S)^*\omega_{FS}+O(n^{-\epsilon})$$ for some small $\epsilon>0$ where by $O(n^{-\epsilon})$ we mean a real closed $(1,1)$ current $T_n$ such that $$|\la T_n,\Phi\ra|\leq Cn^{-\epsilon}\|dd^c\Phi\|_{\infty}$$ where $C>0$ is independent of $n$ and the smooth form $\Phi.$ If $a^{(n)}_j$ are standard complex Gaussian then it is classical that (see \cite{SZ1,SZ2,BloomS}) the identity (\ref{exg}), Proposition \ref{bergman} and Theorem \ref{expected} implies that for $1\leq k\leq m$
  $$\Bbb{E}[\Z_{s_n^1,\dots, s^k_n}]=((\Phi_n^S)^*\omega_{FS})^k\to T_{K,q}^k$$ in the sense of currents as $n\to \infty.$ 
 We prove the analogue result in our setting. We utilize some arguments from \cite{SZ2} to prove the following:
\begin{cor}\label{PL}
Let $X$ be a projective manifold, $L\to X$ be a positive holomorphic line bundle and $K\subset X$ be a locally regular compact set together with a continuous function $q:K\to \Bbb{R}.$ Then
$$\Bbb{E}[\Z_{s_n^1,\dots,s_n^k}]=\Bbb{E}[\Z_{s_n^1}]\wedge \dots\wedge \Bbb{E}[\Z_{s_n^k}].$$
Moreover, $$\Bbb{E}[\Z_{s_n^1,\dots,s_n^k}] \to T_{K,q}^k$$
in the sense of currents as $n\to \infty.$
\end{cor}
\begin{proof}
Note that for systems $(s_n^1,\dots,s_n^k)$ in general position, by Bertini's theorem their zero sets are smooth and intersect transversally and $[\Z_{s^1_n,\dots,s_n^k}]=[\Z_{s_n^1}]\wedge\dots\wedge [\Z_{s_n^k}]$ is a positive closed $(k,k)$ current of mass one and hence $\Bbb{E}[\Z_{s_n^1,\dots,s_n^k}]$ is well-defined. \\ \indent
We prove the assertion by induction on $k.$ Note that the case $k=1$ was proved in Theorem \ref{expected}. Assume that the the claim holds for $k-1.$ We fix $s_1$ such that $X':=Z_{s_1}$ is a smooth hypersurface in $X.$ We also denote the restriction $s':=s|_{X'}$ for generic $s\in \mathcal{S}_n$ and define the restriction map $\chi:\mathcal{S}_n\to\mathcal{S}_n'$ where $\mathcal{S}_n'=\mathcal{S}_n|_{X'}.$ We endow $\mathcal{S}_n'$ with the probability measure $\mu_n':=\chi_*\mu_n.$ Then by induction hypothesis applied on $X'=Z_{s_1}$
 \begin{eqnarray*} 
\int_{\mathcal{S}_n^{k-1}}\la \Z_{s_1,s_2,\dots,s_k},\Phi\ra d\mu_n(s_2)\dots d\mu_n(s_k) &=& \int_{(\mathcal{S}_n')^{k-1}}\la\Z_{s_2',\dots,s'_k},\Phi|_{X'}\ra d\mu_n'(s_2')\dots d\mu_n'(s_k')\\
&=& \la \Bbb{E}[\Z_{s_2'}]\wedge\dots\wedge \Bbb{E}[\Z_{s_k'}],\Phi|_{X'}\ra\\
&=& \int_{Z_{s_1}}\Bbb{E}[\Z_{s_2}]\wedge\dots\wedge \Bbb{E}[\Z_{s_k}]\wedge\Phi
\end{eqnarray*}
then taking the average over $s_1$ we obtain the first assertion. \\ \indent To prove the second assertion, we let
 \begin{equation}\label{an}
 \alpha_n:=\omega+\frac{1}{2n}dd^c\log \|S_n(x,x)\|_{h_n}
 \end{equation} 
and we claim that $$\la \Bbb{E}[\Z_{s_1,\dots,s_k}],\Phi\ra=\la \alpha_n^k,\Phi\ra+ C_{\Phi,n}$$ where $C_{\Phi,n}$ is the ``error term" which satisfies the uniform estimate $$|C_{\Phi,n}|\leq Cn^{-\epsilon} \|dd^c\Phi\|_{\infty} $$
where $\epsilon>0$ small as in Lemma \ref{help} and $C>0$ is independent of smooth form $\Phi$ and sufficiently large $n$. Note that the case $k=1$ was proved in Theorem \ref{expected}. Now, using the above notation and by applying induction hypothesis on $X'=Z_{s_1}$
\begin{eqnarray*}
 \int_{Z_{s_1}}\la\Z_{s_2,\dots,s_k},\Phi\ra d\mu_n(s_2)\dots d\mu_n(s_k) &=& \la \Bbb{E}[\Z_{s_2',\dots,s_k'}],\Phi|_{X'}\ra  \\
 &=& \int_{Z_{s_1}} \alpha_n^{k-1} \wedge \Phi + C_{X',\Phi,n}
\end{eqnarray*}
where $$|C_{X',\Phi,n}|\leq Cn^{-\epsilon}\|dd^c_{x'} (\Phi|_{X'})\|_{\infty}\leq Cn^{-\epsilon}\|dd^c\Phi\|_{\infty}\int_{X'}\omega^{m-1}=Cn^{-\epsilon}\|dd^c\Phi\|_{\infty}$$
where the later equality comes from computing the integral in cohomology. Now, taking the average over $s_1$ and using the estimate in proof of Theorem \ref{expected} we obtain
\begin{eqnarray*}
\la \Bbb{E}[\Z_{s_1,\dots,s_k}],\Phi\ra &=& \la \alpha_n,\alpha_n^{k-1}\wedge\Phi\ra+ C_{\Phi,n}' +\int_{\mathcal{S}_n}C_{X',\Phi,n} d\mu_n(s_1)\\
&=& \la \alpha_n^k,\Phi\ra+ C_{\Phi,n}
\end{eqnarray*}
where 
$$|C_{\Phi,n}|\leq |C_{\Phi,n}'|+\int_{\mathcal{S}_n}|C_{X',\Phi,n}|d\mu_n(s_1)\leq Cn^{-\epsilon}\|\alpha_n^{k-1}\wedge dd^c\Phi\|_{\infty}+Cn^{-\epsilon}\|dd^c{\Phi}\|_{\infty}$$
Thus, the assertion follows from the above estimate and the uniform convergence of Bergman kernels to weighted global extremal function (Proposition \ref{bergman}) together with a theorem of Bedford and Taylor \cite[\S 7]{BT2} on convergence of Monge-Amp\`ere measures.
\end{proof}
\section{Almost everywhere convergence to the expected limit distribution}
%We let $\mathcal{S}_n$  denote the probability space $H^0(X,L^{\otimes n})$ endowed with the probability measure $\mu_n$ defined in the introduction. We also set $\mathcal{S}_{\infty}=\prod_{n=1}^{\infty} \mathcal{S}_n$ endowed with the product measure.
In this section we prove the second assertion in Theorem \ref{main} for the case $k=1$. %We assume that $X$ is projective complex homogeneous manifold. 
\begin{thm} \label{ae1}
Let $X$ be a projective homogeneous manifold, $L\to X$ be a positive holomorphic line bundle and $K\subset X$ be a locally regular compact set with a continuous weight $q.$ Then for $\mu$-a.e. $\{s_n\}_{n\geq 1}\in\mathcal{S}_{\infty}$
$$\widetilde{Z}_{s_n} \to T_{K,q}$$
in the sense of currents as $n\to \infty.$
\end{thm}

\begin{proof}
By \cite[Lemma 3.2.5]{DS11} and Proposition \ref{determine}, it is enough to show that for $\mu$-a.e. $\{s_n\}\in\mathcal{S}_{\infty}$ the sequence of super potentials 
$\{\U_{\Z_{s_n}}\}$ converges to the super-potential of $T_{K,q}$ on smooth measures in $\mathscr{C}_m$. To this end, for a fixed smooth measure $\nu\in\mathscr{C}_m$ we define the sequence of random variables 
$$X_n:\mathcal{S}_{\infty}\to (-\infty,0]$$ 
$$X_n(\{s_j\}_{j\geq 1})=\U_{\Z_{s_n}}(\nu)$$
where $\U_{\Z_{s_n}}$ denotes the super-potential of $\Z_{s_n}$ defined by (\ref{super}). Thus, $\{X_n\}$ is a sequence of negative independent random variables (but they are not identically distributed). Note that since $\nu$ is smooth, $V_{K,q}$ is $\nu$-integrable and by Theorem \ref{expected}
$$\Bbb{E}[X_n]=\int_{\mathcal{S}_n} \la \Z_{s_n},U_{\nu}\ra d\mu_n(s_n)\to \la T_{K,q},U_{\nu}\ra=\U_{T_{K,q}}(\nu)$$ as $n\to \infty$ where $U_{\nu}$ is the quassi-potential of $\nu$ defined by (\ref{number}). This in turn implies that
\begin{equation}\label{mean}
\lim_{n\to\infty}\Bbb{E}[\frac{1}{n}\sum_{k=1}^nX_k]=\U_{T_{K,q}}(\nu).
\end{equation}
On the other hand, the variance of $X_n$ is given by
$$Var[X_n]=\Bbb{E}[X_n^2]-(\Bbb{E}[X_n])^2.$$
Note that the second term in the variance of $X_n$ is bounded by a constant independent of $n$. We will show that the first term is also bounded by a constant independent of $n$: 
\begin{lem}\label{var1}
Let $\nu$ and $X_n$ be as above. Then 
$$Var[X_n]\leq C_{L,\nu}$$ where $C_{L,\nu}>0$ depends only on $\nu$ and $L\to X.$
\end{lem}
\begin{proof}
It is enough to show that $\Bbb{E}[X_n^2]\leq C_{L,\nu}$. Indeed, since $\nu$ is smooth, %by \cite{Kol} there exists a  H\"{o}lder continuous $\phi:X\to \Bbb{R}$ such that $$\nu=(w+dd^c\phi)^m.$$ 
$\U_{\nu}$ is Lipschitz on $\mathscr{C}_1$ with respect to $dist_{\alpha}$. Then by \cite[Proposition A.3]{DS3} and \cite[Theorem 3.9]{DN12} (see also \cite[Theorem 1.1]{DNS}) there exists constants $\alpha>0,C>0$ (depending only on $L\to X$ and $\nu$) such that for every $s_n\in H^0(X,L^{\otimes n})$ and $n\geq 1$ 
$$\int_X\exp(-\alpha\varphi_{s_n})d\nu\leq C$$ where $\varphi_{s_n}$ is the quasi potential of $\Z_{s_n}$ defined by (\ref{choice}). Now, by using $e^x\geq \frac{x^2}{2!}$ for $x\geq 0$ we conclude that  
$$\|\varphi_{s_n}\|_{L^2(\nu)}\leq C_{L,\nu}$$ for some constant $C_{L,\nu}>0$ which depends only on $L\to X$ and $\nu$ but independent of $s_n.$ Thus, by Jensen's inequality we obtain
\begin{eqnarray*}
\Bbb{E}[X_n^2] &=& \int_{\mathcal{S}_n} X_n^2d\mu_n\\
&=& \int_{\mathcal{S}_n} (\U_{\Z_{s_n}}(\nu))^2d\mu_n\\
&=& \int_{\mathcal{S}_n}(\int_X\varphi_{s_n}d\nu)^2d\mu_n\\
&\leq& \int_{\mathcal{S}_n}(\int_X\varphi_{s_n}^2d\nu) d\mu_n\\
&\leq & C_{\omega,\nu}
\end{eqnarray*}
\end{proof}
Hence, by Lemma \ref{var1}, Kolmogorov's strong law of large numbers \cite{Bil} and (\ref{mean}) we obtain that for $\mu$-a.e. $\{s_n\}\in \mathcal{S}_{\infty}$
\begin{equation}\label{avmean}
\frac1n\sum_{k=1}^n\U_{\Z_{s_k}}(\nu)\to\U_{T_{K,q}}(\nu)
\end{equation}
as $n\to \infty.$ 
Note that since $\nu$ is a probability measure, $L^2(\nu)$ norm dominates $L^1(\nu)$ norm. In particular $X_n$'s are bounded. Next, we use the following lemma:
\begin{lem}\cite[Theorem 1.20]{Walters}\label{walters}
Let $\{b_j\}$ be a bounded sequence of negative real numbers. TFAE:
\begin{itemize}
\item[(1)] There exists a subsequence $\{b_{j_k}\}$ of relative density one (i.e. $\frac{k}{j_k}\to 1$ as $k\to \infty$) such that $b_{j_k}\to 0.$
\item[(2)] $\displaystyle\lim_{n\to \infty}\frac1n\sum_{j=1}^nb_j=0.$
\end{itemize}
\end{lem}
Thus, we conclude that $\mu$-a.e. $\{s_n\}\in \mathcal{S}_{\infty}$ has a subsequence $\{s_{n_j}\}$ of density one such that $$\U_{\Z_{s_{n_j}}}(\nu)\to \U_{T_{K,q}}(\nu).$$ We will show that in the above case, in fact, the whole sequence $\{\U_{\Z_{s_n}}(\nu)\}_n$ converges to $\U_{T_{K,q}}(\nu).$ %Assuming the contrary we will derive a contradiction: Suppose that there exists a subsequence $s_{n_k}$ such that $\Z_{s_{n_k}}\not\to T_{K,q}.$ Thus, there exists $\delta>0$ such that 
%\begin{equation}\label{norm}
%|\U_{\Z_{s_{n_k}}}(\omega^m)-\U_{T_{K,q}}(\omega^m)|\geq \delta
%\end{equation}
%for $k\geq 1.$ On the other hand, 
Indeed, since $\U_{\Z_{s_{n}}}\leq0$
 by a variation of  Hartogs Lemma \cite[Prop 3.2.6]{DS11} either $\U_{\Z_{s_{n}}}$ converges uniformly to $-\infty$ or there exists a subsequence $\Z_{s_{n_k}}$ such that $\Z_{s_{n_k}}\to T$ weakly for some $T\in \mathscr{C}_1$ and $\U_{\Z_{s_{n_k}}}\to \U_T$ on smooth measures. However, by Remark \ref{L1} we know that the means $\U_{\Z_{s_{n}}}(\omega^m)$ are uniformly bounded. Hence, the later occurs. Next, we prove the following lemma:
\begin{lem}\label{ext}
For $\mu$-a.e. $\{s_n\}\in \mathcal{S}_{\infty}$
$$\limsup_{n\to \infty}\U_{\Z_{s_n}}(\sigma) \leq \U_{T_{K,q}}(\sigma)$$ for every smooth $\sigma\in\mathscr{C}_m$ where $\U_{T_{K,q}}$ denotes the super-potential of $T_{K,q}$ of mean $\limsup_{n\to \infty}\U_{\Z_{s_n}}(\omega^k).$
In particular, for $\mu$-a.e. $\{s_n\}\in\mathcal{S}_{\infty},$ if $\Z_{s_{n_k}}\to T$ in the sense of currents then $$\U_T\leq \U_{T_{K,q}}$$ on smooth probability measures in $\mathscr{C}_m.$ %where $\U_T$ and $\U_{T{_K,q}}$ are super potentials of $T$ and $T_{K,q}$ respectively. 
\end{lem}

\begin{proof}
For smooth $\sigma\in \mathscr{C}_m$ by (\ref{number}) we have 
\begin{eqnarray*}
\U_{\Z_{s_{n}}}(\sigma) &=& \la \Z_{s_{n}},U_{\sigma}\ra\\
&=&\la\omega+\frac1ndd^c \log\|s_{n}\|_{h_{n}},U_{\sigma}\ra \\
&=& \la \omega,U_{\sigma}\ra+\la\frac1n\log\|s_n\|_{h_{n}},dd^cU_{\sigma}\ra
\end{eqnarray*}
%If $\Z_{s_{n_k}}\to T$ weakly, 
where $U_{\sigma}$ is smooth and $dd^cU_{\sigma}=\sigma-\omega^m.$ %and $\displaystyle\lim_{k\to \infty}\U_{\Z_{s_{n_k}}}(\sigma)=\U_T(\sigma)$ 
%by Proposition \ref{BL} and Fatou's lemma we obtain
Now, using
$$\frac1n \log\|s_n(x)\|_{h_n}\leq \frac{1}{2n}\log\sum_{j=1}^{d_n}|a^{(n)}_j|^2 + \frac{1}{2n}\log\sum_{j=1}^{d_n}\|S_j^n(x)\|_{h_n}^2$$
by Lemma \ref{limsup} and uniform convergence of Bergman kernels in Proposition \ref{bergman} we obtain
\begin{eqnarray*}
\limsup_{n\to \infty}\U_{\Z_{s_n}}(\sigma) &\leq&  \la \omega,U_{\sigma}\ra +\la V_{K,q},dd^cU_{\sigma}\ra\\
&=& \la T_{K,q},U_{\sigma}\ra\\
&=& \U_{T_{K,q}}(\sigma)
\end{eqnarray*}
\end{proof}
Now, by Proposition \ref{cont} the super-potential $\U_{T_{K,q}}$ is continuous on $\mathscr{C}_m.$ If, $\U_T(\nu)\not=\U_{T_{K,q}}(\nu) $ then by \cite[Proposition 3.2.2]{DS11} $$\U_{\Z_{s_{n_k}}}(\nu)<\U_{T_{K,q}}(\nu)$$ for large $k.$ Since $\U_{\Z_{s_{n_k}}}$ are negative this contradicts (\ref{avmean}). Thus, $\U_T$ and $\U_{T_{K,q}}$ agrees on smooth measures. Hence, $T=T_{K,q}$ by Proposition \ref{determine}.\\ \indent
 So far, we have proved that for every smooth measure $\nu\in\mathscr{C}_m$ there exists a set $\mathcal{S}_{\nu}\subset\mathcal{S}_{\infty}$ of probability one such that for every $\{s_n\}_{n\in \Bbb{N}}\in\mathcal{S}_{\nu}$
 $$\U_{\Z_{s_n}}(\nu)\to \U_{T_{K,q}}(\nu)$$ as $n\to \infty.$ Now, we fix a countable dense subset of smooth measures $\{\nu_j\}_{j\in \Bbb{N}} \subset \mathscr{C}_m$ with respect to the $dist_{\alpha}$ for some fixed $\alpha>0$ and define $$\mathscr{S}:=\cap_{j=1}^{\infty}\mathcal{S}_{\nu_j}.$$ Note that $\mathscr{S}$ has probability one. We claim that for every smooth $\nu\in \mathscr{C}_m$ $$\U_{T_{K,q}}(\nu)=\lim_{n\to \infty} \U_{\Z_{s_n}}(\nu)$$ for every $\{s_n\}_{n\geq1}\in\mathscr{S}.$ Indeed, let $\nu' \to \nu$ in $(\mathscr{C}_m,dist_{\alpha})$ with $\nu'\in \{\nu_j\}_{j\in\Bbb{N}}$ then
 $$|\U_{\Z_{s_n}}(\nu)-\U_{T_{K,q}}(\nu)|\leq |\U_{\Z_{s_n}}(\nu-\nu')|+|\U_{\Z_{s_n}}(\nu')-\U_{T_{K,q}}(\nu')|+|\U_{T_{k,q}}(\nu'-\nu)|$$
 where the second term tends to 0 by above argument and the third terms tends to 0 by Proposition \ref{cont}. Finally, since $\nu$ and $\nu'$ are smooth the first term can be bounded $$|\U_{\Z_{s_n}}(\nu-\nu')|=|\la \varphi_{s_n} ,\nu-\nu'\ra|\leq |\int \varphi_{s_n} f dV| $$ where $f:=f_{\nu',\nu}$ is a smooth function with $\|f\|_{\infty} \to 0$ as $\nu'\to \nu$.   
 %by Lemma \ref{ext} we have $$\limsup_{n\to \infty}\U_{\Z_{s_n}}(\nu)\leq \U_{T_{K,q}}(\nu).$$ On the other hand, assume that $\nu_{j_k} \to \nu$ then 
 %$$\U_{T_{K,q}}(\nu_{j_k})= \lim_{n\to \infty} \U_{\Z_{s_n}}(\nu_{j_k})$$
 %for $\{s_n\}\in\mathscr{S}.$ Moreover, by Lemma \ref{cont} $$\U_{T_{K,q}}(\nu)=\displaystyle\lim_{k\to \infty}\U_{T_{K,q}}(\nu_{j_k}).$$ Hence, by Lemma \ref{Climsup}
 %$$\U_{T_{K,q}}(\nu)\leq \limsup_{n\to \infty}\U_{\Z_{s_n}}(\nu)$$ for $\{s_n\}\in \mathscr{S}.$ 
 Since $|\int\varphi_{s_n}dV|\leq C$ for every $s_n\in \mathcal{S}_n$ and $n\in\Bbb{N}$ this finishes the proof.
\end{proof}
\begin{rem}
Finally, we stress that one can work with quasi-potentials of positive closed $(1,1)$ currents rather than super-potentials to prove Theorem \ref{ae1}. In particular, the assertion of Theorem \ref{ae1} is still valid if $X$ is merely projective manifold but not homogenous. We refer the reader to \cite{BloomL} for such a treatment.  
\end{rem}
\section{Almost everywhere convergence for bidegree $(k,k)$}
 Let $S^k_n=(s_n^1,s_n^2,\dots,s_n^k)$ be a $k$-tuple of i.i.d random holomorphic sections $s_n^j\in \mathcal{S}_n$ for $j=1,2,\dots,k$ where $1\leq k\leq m.$ We are interested in distribution of simultaneous zeros: $$Z_{S^k_n}:=\{x\in X: s_n^1(x)=\dots=s_n^k(x)=0\}$$ We denote set of all such $k$-tuples (respectively sequences of $k$-tuples) by $\mathcal{S}_n^k:=\prod_{j=1}^k\mathcal{S}_n$ (respectively by $\mathcal{S}_{\infty}^k:=\prod_{j=1}^k \mathcal{S}_{\infty}$) endowed with $k$-fold the product measure $\mu_n^k$ (respectively $\mu^k$) induced by $\mu_n$ (respectively $\mu$). By Bertini's theorem \cite[pp 137]{GH} with probability one the zero sets of $Z_{s^j_n}$ are smooth and intersect transversally. In particular, for generic $S^k_n\in \mathcal{S}_n^k$ the zero set $Z_{S^k_n}$ is a complex submanifold of codimension $k.$ Moreover, almost surely the current of integration along the set $Z_{S^k_n}$ is given by 
 $$[Z_{S^k_n}]=[Z_{s_n^1}]\wedge \dots \wedge [Z_{s_n^k}].$$ 
 Next, we may write
 \begin{equation}\label{potential}
 \Z_{S^k_n}:=\frac{1}{n^k}[Z_{S^k_n}]=\omega^k+dd^cU_{S^k_n}
 \end{equation}
 where 
 \begin{equation}\label{superk}
 U_{S^k_n}(z)=\int_{z\not=\zeta}(\Z_{S^k_n}(\zeta)-\omega^k(\zeta))\wedge K(z,\zeta) 
 \end{equation}
  is the negative $(k-1,k-1)$ current given by Theorem \ref{DS}. 
  
 \begin{thm}\label{ae2}
 Let $X$ be a projective homogeneous manifold, $L\to X$ be a positive holomorphic line bundle and $K\subset X$ be a locally regular compact set with a continuous weight $q$. Then for $\mu^k$-a.e. $\{S^k_n\}_{n\geq1}\in \mathcal{S}^k_{\infty}$ 
 $$\Z_{S^k_n}\to T_{K,q}^k$$
 in the sense of currents as $n\to \infty.$
 \end{thm}
\begin{proof}
 We will prove the theorem by induction on $k$. Note that the case $k=1$ was proved in Theorem \ref{ae1}. Let's assume that the assertion holds for $k-1.$\\ \indent
 Now, given $\Z_{S^k_n}\in \mathcal{S}_n^k$ we let $U_{\Z_{S^k_n}}$ be as in (\ref{superk}) then by Theorem \ref{DS} 
\begin{equation}\label{5.3}
|\langle U_{\Z_{S^k_n}},\omega^{m-k+1}\rangle|\leq C
\end{equation}
where $C>0$ is independent of $S^k_n$ and $n\in \Bbb{N}.$ We denote the super-potential of $\Z_{S^k_n}$ by  
$$\U_{\Z_{S^K_n}}(R)=\la U_{\Z_{S^k_n}},R\ra$$ whenever $R$ is a smooth form in $\mathscr{C}_{m-k+1}.$
Note that by \cite[Lemma 3.2.5]{DS11} and Proposition \ref{determine} it is enough to show that with probability one, $\U_{\Z_{S^k_n}}(R)$ converges to $\U_{T^k_{K,q}}(R)$ for every smooth form $R\in \Cmk.$ To this end we fix a smooth form $R\in \Cmk$ and define the random variables 
$$X_n:\mathcal{S}_{\infty}^k\to (-\infty,0]$$ 
$$X_n(\{S^k_j\}_{j\geq 1})=\U_{\Z_{S_n}}(R)$$
Note that $X_n$ are independent (but not identically distributed) random variables. We will need the following lemma, proof of which is deferred until the end of this section.
\begin{lem}\label{var2}
Let $X_n$ be as above then the variance of $X_n$ satisfies $$Var[X_n]\leq C_{L,R}n^{-\epsilon}$$ where $\epsilon>0$ small and $C_{L,R}>0$ depends only on the Chern class of the line bundle $L\to X$ and the smooth form $R.$
\end{lem}
 Then by Kolmogorov's strong law of large numbers we conclude that for $\mu^k$-a.e. $S\in \mathcal{S}_{\infty}^{k}$ 
 $$\frac1n\sum_{j=1}^nX_j(S)- \Bbb{E}[\frac1n\sum_{j=1}^nX_j]\to 0$$
 as $n\to \infty.$ On the other hand, by Corollary \ref{PL} 
 $$\Bbb{E}[\Z_{S^k_n}]\to T_{K,q}^k$$
 in the sense of currents as $n\to \infty.$ Thus, we infer that with probability one 
\begin{equation}\label{mean2}
 \lim_{n\to\infty}\frac{1}{n}\sum_{j=1}^n\U_{\Z_{S_j}}(R)=\U_{T_{K,q}}(R).
 \end{equation}
Note that since $R$ is smooth by (\ref{5.3}) $$|X_n(\{S^k_j\}_{j\geq1})|\leq C_R$$  where $C_R>0$ depends on $R$ but independent of $S^k_n\in \mathcal{S}_n^k$ and $n\in\Bbb{N}.$ Thus, by Lemma \ref{walters} we conclude that $\mu^k$-a.e. $\{S^k_n\}_{n\geq1}$ has a subsequence $\{S^k_{n_j}\}$ of density one such that 
 $$\U_{\Z_{S^k_{n_j}}}(R)\to \U_{T_{K,q}^k}(R)$$ as $j\to \infty.$ Next, we will show that in this case in fact the whole sequence $\{\U_{\Z_{S^k_n}}(R)\}_{n\geq1}$ converges. %Assuming the contrary i.e. $\U_{\Z_{S_n}}(R)\not\to\U_{T^k_{K,q}}(R)$ we will derive a contradiction. 
Indeed, since $\U_{\Z_{S^k_n}}\leq 0$ on $\Cmk$ by \cite[Prop 3.2.6]{DS11} either $\{\U_{\Z_{S^k_n}}\}_{n\geq 1}$ converges uniformly to $-\infty$ or $\{S^k_n\}_{n\geq1}$ has a subsequence $S^k_{n_i}$ such that $$\Z_{S^k_{n_i}} \to T$$ for some $T\in \Ck$  in the sense of currents as $n_i\to \infty$ and $$\U_{\Z_{S^k_{n_i}}}\to \U_T$$ on smooth forms in $\Cmk$ where $\U_T$ is the super-potential of $T$ of mean $m:=\lim_{n_i\to \infty}\la U_{\Z_{S^k_{n_i}}},\omega^{m-k+1}\ra.$ However, the former is not possible as the means $\{\la U_{\Z_{S^k_n}},\omega^{m-k+1}\ra\}_{n\geq1}$ are uniformly bounded by (\ref{5.3}). Hence, the later occurs. %We will show that $\U_T=\U_{T_{K,q}^k}$ on smooth forms in $\Cmk$ and by Proposition \ref{determine} this implies that $T=T_{K,q}^k$ which is a contradiction.   
Next, we prove the following lemma:
\begin{lem}
For $\mu^k$-a.e. $\{S_n\}\in\mathcal{S}_{\infty}^k,$
$$\limsup_{n\to \infty}\U_{\Z_{S_n}}(\Phi)\leq \U_{T_{K,q}^k}(\Phi)$$
for every smooth form $\Phi\in\Cmk$ where $\U_{T^k_{K,q}}$ denotes the super-potential of $T^k_{K,q}$ of mean $\limsup_{n\to \infty}\U_{\Z_{S_n}}(\omega^{m-k+1}).$
In particular, for $\mu^k$-a.e. $\{S^k_n\}\in\mathcal{S}_{\infty}^k,$ if $\Z_{S^k_{n_i}}\to T$ for some $T\in \Ck$ in the sense of currents then
$$\U_T\leq \U_{T_{K,q}^k}$$ on smooth forms in $\Cmk.$
\end{lem}
\begin{proof}
 We let $\U$ denote the super-potential of mean zero. We prove the lemma by induction on $k$.  Note that the case $k=1$ was proved in Lemma \ref{ext}. Assume that the assertion holds for $k-1.$ By Bertini's theorem \cite{GH} for generic  $S^k_n=(S'_n,s^k_n)$ we may write
$$\Z_{S^k_n}= \Z_{S'_n}\wedge \Z_{s^k_n}$$ 
where $S'_n\in\mathcal{S}_{n}^{k-1}.$  Then by (\ref{wedge}) almost surely the super potential of $\Z_{S^k_{n}}$ of mean zero is given by 
$$\U_{\Z_{S^k_{n}}}(\Phi)=\la \Z_{S^{'}_{n}},\omega\wedge U_{\Phi}\ra+\U_{\Z_{s_{n}^k}}(\Z_{S^{'}_{n}}\wedge dd^c U_{\Phi})$$
where $U_{\Phi}$ is a smooth quasi-potential of $\Phi\in\mathscr{C}_{m-k+1}$ of mean zero. Moreover, by induction hypothesis for generic sequences $\Z_{S'_n} \to T_{K,q}^{k-1}$ and $Z_{s^k_n}\to T_{K,q}$ in the sense of currents. Then by Lemma \ref{Climsup} we have
%\begin{eqnarray*}
$$\limsup_{n\to \infty}\U_{\Z_{s^k_{n}}}(\Z_{S'_{n}}\wedge dd^c U_{\Phi}) \leq \U_{T_{K,q}}(T_{K,q}^{k-1}\wedge dd^cU_{\Phi})$$
%\\
%&=&\U_{T_{K,q}^{k-1}}(dd^cV_{K,q}\wedge U_{\Phi})
%\end{eqnarray*} %where $\U_{T_{K,q}^{k-1}}$ is the super-potential of $T_{K,q}^{k-1}$ of mean zero. 
On the other hand, by induction hypothesis again $$\limsup_{n\to\infty}\U_{S_n'}(\omega\wedge\Phi)=\limsup_{n\to \infty}\la \Z_{S'_{n}},\omega\wedge U_{\Phi}\ra\leq \U_{T_{K,q}^{k-1}}(\omega\wedge \Phi)=\la T_{K,q}^{k-1},\omega\wedge U_{\Phi}\ra$$ %Since $dd^cU_R=R-\omega^{m-k}$ by Lemma \ref{limsup}
%$$\limsup_{n\to\infty}\la\frac1n\log \|s_{n_j}^k\|_{h_{n_j}},\omega^{k-1}\wedge dd^cU_R\ra\leq \la V_{K,q},\omega^{k-1}\wedge dd^cU_R\ra.$$
Thus, combining these and using (\ref{wedge}) we conclude that
$$\limsup_{n\to \infty}\U_{\Z_{S_n}}(\Phi)\leq \U_{T_{K,q}^k}(\Phi)$$ for every smooth form $\Phi$ in $\Cmk.$
\end{proof}
Now, by Proposition \ref{cont}, the super-potential $\U_{T_{K,q}}$ is continuous on $\Cmk$. If $\U_T(R)\not=\U_{T_{K,q}^k}(R)$ then $$\U_T(R)<U_{T_{K,q}^k}(R)$$ thus, by \cite[Prop 3.2.2]{DS11} $$\U_{\Z_{S^k_{n_j}}}(R)<\U_{T_{K,q}^k}(R)$$ for large $n_j$ but this contradicts (\ref{mean2}) as $\U_{Z_{n_j}}$ are negative. Hence, we conclude that for every smooth $R\in \mathscr{C}_{m-k+1}$ there exists a set $\mathcal{S}_R\subset \mathcal{S}_{\infty}^k$ of probability one such that $$\U_{T_{K,q}^k}(R)=\lim_{n\to \infty}\U_{Z_{S^k_n}}(R)$$ for every $\{S^k_n\}_{n\geq1}\in \mathcal{S}_R.$
Finally, applying the density argument given at the end of the Proof of Theorem \ref{ae1} completes the proof.
\end{proof}
Next, we prove Lemma \ref{var2}. The proof is based on induction on bidegree. For $k=1,$ we provide a different proof than the one given in Lemma \ref{var1} which provides some precision on the bound of the variance of $X_n's$.  We utilize some ideas from \cite{SZ, Shiffman}.
\begin{proof}[Proof of Lemma \ref{var2}]
First, we prove the case $k=1:$ we use the same notation as in Lemma \ref{help} and Theorem \ref{expected}. Note that 
$$Var[X_n]=\Bbb{E}[X_n^2]-(\Bbb{E}[X_n])^2$$ %and by Theorem \ref{expected} it is enough to estimate the term %$E[X_n^2]$ where
where
$$\Bbb{E}[X_n^2]=\int_{\mathcal{S}_n}\la \Z_s,U_{\nu}\ra^2d\mu_n(s)$$ and $\nu$ is a smooth measure in $\mathscr{C}_m$ and $U_{\nu}$ is a smooth quasi-potential as in (\ref{number}).
We claim that 
$$\Bbb{E}[X_n^2]=\la\alpha_n,U_{\nu}\ra^2+O(n^{-\epsilon}\|dd^cU_{\nu}\|_{\infty})$$
where $\alpha_n$ as in (\ref{an}) and $\epsilon>0$ as in Lemma \ref{help}. On the other hand, by Theorem \ref{expected} 
$$\Bbb{E}[X_n]=\la\alpha_n,U_{\nu}\ra+O(n^{-\epsilon}\|dd^cU_{\nu}\|_{\infty})$$ hence, the assertion follows for $k=1.$ Next, we prove the claim. Following \cite{SZ} we write
$$\Bbb{E}[X_n^2]=\frac{1}{n^2}\int_X\int_Xdd^cU_{\nu}(z)dd^cU_{\nu}(w)\int_{\C^{d_n}}\log |\la a,f(z)\ra|\log|\la a,f(w)\ra| d\probn(a)$$ 
The later integrant can be written as 
\begin{eqnarray*}
\log|\la a,f(z)\ra|\log|\la a,f(w)\ra|=\log|f(z)| \log|f(w)| &+& \log|f(z)|\log|\la a,u(w)\ra| \\ &+&  \log|f(w)|\log|\la a,u(z)\ra| 
 +  \log|\la a,u(z)\ra|\log|\la a,u(w)\ra|
\end{eqnarray*}
where $f$ and $u$ as in the proof of Theorem \ref{expected}. Thus, we may write
$$\Bbb{E}[X_n^2]=: I_1(n)+I_2(n)+I_3(n)+I_4(n).$$
It follows from Theorem \ref{expected} that
$$I_1(n)=\la \alpha_n,U_{\nu}\ra^2$$ and for $j=2,3$
$$|I_j(n)|\leq Cn^{-\epsilon}|\la \alpha_n,U_{\nu}\ra|\|dd^cU_{\nu}\|_{\infty}.$$
Finally, we claim that 
\begin{equation}\label{i4}
|I_4(n)|\leq C_m n^{-2\epsilon}\|dd^cU_{\nu}\|_{\infty}^2
\end{equation} where $C_m>0$ is independent of $n.$ Indeed, by Cauchy-Schwarz inequality 
$$|I_4(n)|\leq\frac{1}{n^2}  \|dd^cU_{\nu}\|_{\infty}^2\sup_{z\in X}\int_{\C^{d_n}}(\log|\la a,u(z)\ra|)^2 d\probn(a)$$
thus, it is enough to show that for every unit vector $u\in \C^{d_n}$ 
$$\int_{\{(\log|\la a,u\ra|)^2>mn^{2-2\epsilon}\}}(\log|\la a,u\ra|)^2 d\probn(a)\leq C_mn^{1-\epsilon}.$$
where $C_m$ depends only on $m.$ We let
$$L_j:=\{a\in \C^{d_n}: jn^{2-2\epsilon}< (\log|\la a,u\ra|)^2\leq (j+1)n^{2-2\epsilon}\}.$$
Note that $$L_j\subset R_{\sqrt{j}} \cup D_{\sqrt{j}}$$ where $R_{\sqrt{j}}$ and $D_{\sqrt{j}}$ as in the proof of Lemma \ref{help}. Then by (\ref{rk}) and (\ref{dk}) we have   
$$\probn(L_j)\leq r_{\sqrt{j}}(n)-r_{\sqrt{j+1}}(n) +C_mn^me^{-2\sqrt{j}n^{1-\epsilon}} $$
which implies that for sufficiently large $n$
\begin{eqnarray*}
\int_{\{(\log|\la a,u\ra|)^2>mn^{2-2\epsilon}\}}(\log|\la a,u\ra|)^2 d\probn(a) &\leq & \sum^{\infty}_{j=m}(j+1)n^{2-2\epsilon}[(r_{\sqrt{j}}(n)-r_{\sqrt{j+1}}(n)+C_mn^me^{-2\sqrt{j}n^{1-\epsilon}}]\\
&\leq& n^{2-2\epsilon}[(m+1)r_{\sqrt{m}}(n)+\sum_{j=m+1}^{\infty}r_{\sqrt{j}}(n)+C_mn^m\sum_{j=m}^{\infty}(j+1) e^{-2\sqrt{j}n^{1-\epsilon}}]\\
&\leq& C_mn^{m+2-2\epsilon}[ \frac{1}{n^{(1-\epsilon)\rho}}+\sum_{j=m+1}^{\infty}\frac{1}{j^{\frac{\rho}{2}}n^{(1-\epsilon)\rho}}+ \sum_{j=m+1}^{\infty}(j+1)e^{-2\sqrt{j}n^{1-\epsilon}}]
\end{eqnarray*}
Thus, using $ (\rho-1)(1-\epsilon)\geq m$ we obtain
$$\int_{\{(\log|\la a,u\ra|)^2>mn^{2-2\epsilon}\}}(\log|\la a,u\ra|)^2 d\probn(a) \leq C_m n^{1-\epsilon}.$$
Since, $\probn$ is a probability measure we conclude that 
$$\int_{\C^{d_n}}(\log|\la a,u\ra|)^2 d\probn(a)\leq C_mn^{2-2\epsilon} $$ 
which proves (\ref{i4}) and this completes the proof of the case $k=1$. \\ \indent
  Now, we assume that the assertion holds for $k-1.$ We denote $S^k_n=(S_n',s_k)$ where\\ $S_n'=(s_1,\dots,s_{k-1})\in \mathcal{S}_n^{k-1}.$ Then by Corollary \ref{PL} 
$$\Bbb{E}[\Z_{S^k_n}]=\Bbb{E}[\Z_{S'_{n}}]\wedge \Bbb{E}[\Z_{s^k_n}].$$ and by (\ref{sym})
\begin{eqnarray*}
\U_{\Z_{S^k_n}}(R) %&=&\la U_{\Z_{S_n}},R\ra\\
= \la \Z_{S'_n}\wedge\Z_{s^k_n}, U_R\ra
\end{eqnarray*}
where $U_R$ is a smooth quasi-potential of the smooth form $R\in\mathscr{C}_{m-k+1}$ of mean $\la U_{S^k_n},\ommk\ra.$ Note that
\begin{eqnarray*}
Var[X_n] &=& \Bbb{E}[X_n^2]-(\Bbb{E}[X_n])^2 \\
&=& \int_{\mathcal{S}_n^k}(\la \Z_{S_n},U_R\ra)^2d\mu_n^k-(\la \Bbb{E}[\Z_{S'_n}]\wedge \Bbb{E}[\Z_{s_n^k}],U_R\ra)^2
\end{eqnarray*} 
We define $J_1$ and $J_2$ such that
$$J_1+J_2:=\la \Z_{S^k_n},U_R\ra^2-\la \Bbb{E}[\Z_{S'_n}]\wedge \Bbb{E}[\Z_{s_n^k}],U_R\ra^2$$ where
$$J_1(S'_n,s^k_n)=\la \Z_{S^k_n},U_R\ra^2-\la\Z_{S'_n}\wedge \Bbb{E}[\Z_{s^k_n}],U_R\ra^2$$
$$J_2(S'_n)=\la\Z_{S'_n}\wedge \Bbb{E}[\Z_{s^k_n}],U_R\ra^2-\la \Bbb{E}[\Z_{S'_n}]\wedge \Bbb{E}[\Z_{s_n^k}],U_R\ra^2$$
which are well-defined for a.e. $S_n'$ and $s^k_n$ (see proof of Corollary \ref{PL}). Note that 
$$Var[X_n] =\Bbb{E}[J_1]+\Bbb{E}[J_2]$$ 
\\ \indent
For a generic $S_n'\in \mathcal{S}_n^{k-1}$ the set  $V:=\{S'_n=0\}$ is a complex submanifold (not necessarily homogeneous!) of codimension $k-1.$ Moreover, $\Z_{S'_n}\wedge \Z_{s^k_n}=\Z_{s^k_n}|_{V}$ for generic $s_n^k.$ Thus,  
$$\U_{\Z_{S^k_n}}(R)=\la\Z_{s^k_n}|_V,U_R|_V\ra$$
Then applying the induction hypothesis with $k=1$ to $(V,\frac{1}{\|R|_V\|}R|_V)$ and $\mu_n'$ in place of $(X,\nu)$ and $\mu_n$ where $\rho:\mathcal{S}_n\to \mathcal{S}_n'$ is the restriction map and $\rho_*\mu_n=\mu_n',$ we obtain
\begin{eqnarray*}
\int_{\mathcal{S}_n} J_1(S'_n,s^k_n)d\mu_n(s_n^k) &=& \int_{\mathcal{S}_n}\la\Z_{s_n^k}|_V,U_R|_V\ra^2 d\mu_n'(s_n^k)-\la \Bbb{E}[\Z_{s_n^k}|_V],U_R|_V\ra^2\\
&\leq& C\|R|_V\|^2n^{-\epsilon}\|dd^c(U_R|_V)\|_{\infty}\\
&\leq& C_L\|R\|^2n^{-\epsilon}\|dd^cU_R\|_{\infty}
\end{eqnarray*}
where $\|R|_V\|$ (respectively $\|R\|$) denotes the mass of restriction of $R$ on $V$ (respectively the mass of $R$ on $X$) with respect to $\omega|_{V}$ (respectively $\omega$). Thus, taking the average over $S_n'$ we obtain
$$\Bbb{E}[J_1]=\int_{\mathcal{S}_n^{k-1}}\int_{\mathcal{S}_n} J_1(S_n',s^k_n)d\mu_n(s^k_n)_n d\mu_n^{k-1}(S'_n)\leq C_{L,R}n^{-\epsilon}.$$
On the other hand, for a random variable $Y_n$ of mean $m$ we have $\Bbb{E}[(Y_n-m)^2]=\Bbb{E}[(Y_n)^2]-m^2.$ Applying this argument to the random variables $$Y_n(\{S_j'\}_{j\geq1}):=\la \Z_{S_n'}\wedge \Bbb{E}[\Z_{s_n^k}],U_R\ra$$ we obtain
\begin{eqnarray*}
\Bbb{E}[J_2] &=& \int_{\mathcal{S}_{n}^{k-1}} ( \la \Z_{S'_n}\wedge \Bbb{E}[\Z_{s^k_n}],U_R\ra^2d\mu_n^{k-1}(S'_n)-\la \Bbb{E}[\Z_{S'_n}]\wedge \Bbb{E}[\Z_{s^k_n}],U_R\ra^2\\
&=& \int_{\mathcal{S}_{n}^{k-1}} \big(\la \Z_{S'_n}\wedge \Bbb{E}[\Z_{s^k_n}]-\Bbb{E}[\Z_{S'_n}]\wedge \Bbb{E}[\Z_{s^k_n}],U_R\ra\big)^2d \mu_n^{k-1}(S'_n)\\
&=& \int_{\mathcal{S}_{n}^{k-1}} \big[\int_{\mathcal{S}_n}(\la(\Z_{S'_n}-\Bbb{E}[\Z_{S'_n}])\wedge \Z_{s^k_n},U_R\ra d\mu_n(s^k_n)\big]^2 d\mu_n^{k-1}(S'_n)\\
&\leq&  \int_{\mathcal{S}_n} \int_{\mathcal{S}_{n}^{k-1}} ( \la (\Z_{S'_n}-\Bbb{E}[\Z_{S'_n}])\wedge \Z_{s^k_n},U_R\ra^2d\mu_n^{k-1}(S'_n)d\mu_n(s^k_n)
\end{eqnarray*}
where the last inequality follows from Jensen's inequality and Fubini's theorem. Now, letting $W:=\{s^k_n=0\}$ since $W$ is a smooth hypersurface for generic $s_n^k$ we have
$$\la(\Z_{S'_n}-\Bbb{E}[\Z_{S'_n}])\wedge \Z_{s^k_n}, U_R\ra= \la (Z_{S'_n}-\Bbb{E}[\Z_{S'_n}])|_W,U_R|_W\ra$$
and applying the induction hypothesis for $k-1$ on $W$ we obtain that 
$$\Bbb{E}[J_2]\leq C_{L,R}n^{-\epsilon}\|dd^cU_R\|_{\infty}$$
and this finishes the proof.
\end{proof}

 \bibliographystyle{plain}
\bibliography{5910}

\begin{thebibliography}{10}

\bibitem{Abreu}
M.~Abreu.
\newblock K\"ahler geometry of toric manifolds in symplectic coordinates.
\newblock In {\em Symplectic and contact topology: interactions and
  perspectives ({T}oronto, {ON}/{M}ontreal, {QC}, 2001)}, volume~35 of {\em
  Fields Inst. Commun.}, pages 1--24. Amer. Math. Soc., Providence, RI, 2003.

\bibitem{BT2}
E.~Bedford and B.~A. Taylor.
\newblock A new capacity for plurisubharmonic functions.
\newblock {\em Acta Math.}, 149(1-2):1--40, 1982.

\bibitem{BBN}
R.~Berman, S.~Boucksom, and D.~Witt~Nystr{\"o}m.
\newblock Fekete points and convergence towards equilibrium measures on complex
  manifolds.
\newblock {\em Acta Math.}, 207(1):1--27, 2011.

\bibitem{Berman}
R.~J. Berman.
\newblock Bergman kernels and equilibrium measures for line bundles over
  projective manifolds.
\newblock {\em Amer. J. Math.}, 131(5):1485--1524, 2009.

\bibitem{Bil}
P.~Billingsley.
\newblock {\em Probability and measure}, volume 939.
\newblock Wiley, 2012.

\bibitem{Bloom1}
T.~Bloom.
\newblock Random polynomials and {G}reen functions.
\newblock {\em Int. Math. Res. Not.}, (28):1689--1708, 2005.

\bibitem{BloomL}
T.~Bloom and N.~Levenberg.
\newblock Random {P}olynomials and {P}luripotential-{T}heoretic {E}xtremal
  {F}unctions.
\newblock {\em Potential Anal.}, 42(2):311--334, 2015.

\bibitem{BloomS}
T.~Bloom and B.~Shiffman.
\newblock Zeros of random polynomials on {$\Bbb C^m$}.
\newblock {\em Math. Res. Lett.}, 14(3):469--479, 2007.

\bibitem{BM}
S.~Bochner and D.~Montgomery.
\newblock Groups on analytic manifolds.
\newblock {\em Ann. of Math. (2)}, 48:659--669, 1947.

\bibitem{BR}
A.~Borel and R.~Remmert.
\newblock \"{U}ber kompakte homogene {K}\"ahlersche {M}annigfaltigkeiten.
\newblock {\em Math. Ann.}, 145:429--439, 1961/1962.

\bibitem{BGS}
J.-B. Bost, H.~Gillet, and C.~Soul{\'e}.
\newblock Heights of projective varieties and positive {G}reen forms.
\newblock {\em J. Amer. Math. Soc.}, 7(4):903--1027, 1994.

\bibitem{CMM}
D.~Coman, X.~Ma, and G.~Marinescu.
\newblock Equidistribution for sequences of line bundles on normal kaehler
  spaces.
\newblock {\em arXiv preprint arXiv:1412.8184}, 2014.

\bibitem{Delzant}
T.~Delzant.
\newblock Hamiltoniens p\'eriodiques et images convexes de l'application
  moment.
\newblock {\em Bull. Soc. Math. France}, 116(3):315--339, 1988.

\bibitem{DemBook}
J.-P. Demailly.
\newblock {\em Complex analytic and differential geometry.}
\newblock http://www-fourier.ujf-grenoble.fr/~demailly/manuscripts/agbook.pdf,
  2009.

\bibitem{DS3}
T.-C. Dinh and N.~Sibony.
\newblock Distribution des valeurs de transformations m\'eromorphes et
  applications.
\newblock {\em Comment. Math. Helv.}, 81(1):221--258, 2006.

\bibitem{DS11}
T.-C. Dinh and N.~Sibony.
\newblock Super-potentials of positive closed currents, intersection theory and
  dynamics.
\newblock {\em Acta Math.}, 203(1):1--82, 2009.

\bibitem{DS10}
T.-C. Dinh and N.~Sibony.
\newblock Super-potentials for currents on compact {K}\"ahler manifolds and
  dynamics of automorphisms.
\newblock {\em J. Algebraic Geom.}, 19(3):473--529, 2010.

\bibitem{DN12}
T.C. Dinh and V.A. Nguyen.
\newblock Characterization of monge-ampere measures with holder continuous
  potentials.
\newblock {\em arXiv preprint arXiv:1204.4883}, 2012.

\bibitem{DNS}
T.C. Dinh, V.A. Nguy{\^e}n, and N.~Sibony.
\newblock Exponential estimates for plurisubharmonic functions.
\newblock {\em Journal of Differential Geometry}, 84(3):465--488, 2010.

\bibitem{GiS}
H.~Gillet and C.~Soul{\'e}.
\newblock Arithmetic intersection theory.
\newblock {\em Inst. Hautes \'Etudes Sci. Publ. Math.}, (72):93--174 (1991),
  1990.

\bibitem{GH}
P.~Griffiths and J.~Harris.
\newblock {\em Principles of algebraic geometry}.
\newblock Wiley Classics Library. John Wiley \& Sons Inc., New York, 1994.
\newblock Reprint of the 1978 original.

\bibitem{GZ}
V.~Guedj and A.~Zeriahi.
\newblock Intrinsic capacities on compact {K}\"ahler manifolds.
\newblock {\em J. Geom. Anal.}, 15(4):607--639, 2005.

\bibitem{Ham}
J.~M. Hammersley.
\newblock The zeros of a random polynomial.
\newblock In {\em Proceedings of the {T}hird {B}erkeley {S}ymposium on
  {M}athematical {S}tatistics and {P}robability, 1954--1955, vol. {II}}, pages
  89--111, Berkeley and Los Angeles, 1956. University of California Press.

\bibitem{IZ}
I.~Ibragimov and D.~Zaporozhets.
\newblock On distribution of zeros of random polynomials in complex plane.
\newblock In {\em Prokhorov and Contemporary Probability Theory}, pages
  303--323. Springer, 2013.

\bibitem{Kac}
M.~Kac.
\newblock On the average number of real roots of a random algebraic equation.
\newblock {\em Bull. Amer. Math. Soc.}, 49:314--320, 1943.

\bibitem{Klimek}
M.~Klimek.
\newblock {\em Pluripotential theory}, volume~6 of {\em London Mathematical
  Society Monographs. New Series}.
\newblock The Clarendon Press, Oxford University Press, New York, 1991.
\newblock Oxford Science Publications.

\bibitem{NgZ}
T.~V. Nguyen and A.~Z{\'e}riahi.
\newblock Familles de polyn\^omes presque partout born\'ees.
\newblock {\em Bull. Sci. Math. (2)}, 107(1):81--91, 1983.

\bibitem{SaffTotik}
E.~B. Saff and V.~Totik.
\newblock {\em Logarithmic potentials with external fields}, volume 316 of {\em
  Grundlehren der Mathematischen Wissenschaften [Fundamental Principles of
  Mathematical Sciences]}.
\newblock Springer-Verlag, Berlin, 1997.
\newblock Appendix B by Thomas Bloom.

\bibitem{Shiffman}
B.~Shiffman.
\newblock Convergence of random zeros on complex manifolds.
\newblock {\em Sci. China Ser. A}, 51(4):707--720, 2008.

\bibitem{SZ}
B.~Shiffman and S.~Zelditch.
\newblock Distribution of zeros of random and quantum chaotic sections of
  positive line bundles.
\newblock {\em Comm. Math. Phys.}, 200(3):661--683, 1999.

\bibitem{SZ3}
B.~Shiffman and S.~Zelditch.
\newblock Equilibrium distribution of zeros of random polynomials.
\newblock {\em Int. Math. Res. Not.}, (1):25--49, 2003.

\bibitem{SZ1}
B.~Shiffman and S.~Zelditch.
\newblock Random polynomials with prescribed {N}ewton polytope.
\newblock {\em J. Amer. Math. Soc.}, 17(1):49--108 (electronic), 2004.

\bibitem{SZ2}
B.~Shiffman and S.~Zelditch.
\newblock Number variance of random zeros on complex manifolds.
\newblock {\em Geometric and Functional Analysis}, 18(4):1422--1475, 2008.

\bibitem{Siciak}
J.~Siciak.
\newblock Extremal plurisubharmonic functions in {${\bf C}^{n}$}.
\newblock {\em Ann. Polon. Math.}, 39:175--211, 1981.

\bibitem{Triebel}
H.~Triebel.
\newblock {\em Interpolation theory, function spaces, differential operators},
  volume~18 of {\em North-Holland Mathematical Library}.
\newblock North-Holland Publishing Co., Amsterdam, 1978.

\bibitem{Walters}
P.~Walters.
\newblock {\em An introduction to ergodic theory}, volume~79 of {\em Graduate
  Texts in Mathematics}.
\newblock Springer-Verlag, New York, 1982.

\end{thebibliography}
\end{document}